\documentclass[12pt]{amsart}
\usepackage{epsfig,comment}
\usepackage{mathtools}
\usepackage[subnum]{cases}
\usepackage{enumerate}
\usepackage{bbm,bm}
\usepackage{amsmath,amssymb,mathrsfs}
\usepackage{color}
\usepackage[bookmarks=true,
bookmarksnumbered=true, breaklinks=true,
pdfstartview=FitH, hyperfigures=false,
plainpages=false, naturalnames=true,
colorlinks=true,pagebackref=true,
pdfpagelabels,colorlinks=true,
	linkcolor=blue,
	citecolor=black,
	filecolor=black,
	urlcolor=black]{hyperref}
\usepackage{wasysym}

\usepackage{pgfplots}
\pgfplotsset{width=10cm,compat=1.9}
\usepgfplotslibrary{fillbetween}

\usepgfplotslibrary{external}
\newtheorem{theorem}{Theorem}[section]
\newtheorem{definition}[theorem]{Definition}
\newtheorem{lemma}[theorem]{Lemma}
\newtheorem{corollary}[theorem]{Corollary}
\newtheorem{remark}[theorem]{Remark}

\newtheorem{conjecture}[theorem]{Conjecture}
\newtheorem{problem}[theorem]{Problem}

\newcommand{\R}{\mathbb{R}}
\newcommand{\K}{\mathcal{K}} 
 
\newcommand{\Sp}{\mathbb{S}^{n}}
\newcommand{\HH}{\mathbb{H}}

\DeclareMathOperator{\conv}{conv}

\DeclareMathOperator{\vol}{vol}
\DeclareMathOperator{\proj}{proj}

\DeclareMathOperator{\dist}{dist}

\DeclareMathOperator{\area}{area}
\DeclareMathOperator{\perim}{perim}
\DeclareMathOperator{\bd}{bd}
\DeclareMathOperator{\Sph}{\mathbb{S}}
\DeclareMathOperator{\diam}{diam}
\DeclareMathOperator{\inter}{int}
\DeclareMathOperator{\as}{as}
\DeclareMathOperator{\del}{del}

\title{Steiner symmetrization on the sphere}

\author[B. Basit]{Bushra Basit}
\author[S. Hoehner]{Steven Hoehner}
\author[Z. L\'angi]{Zsolt L\'angi}
\author[J. Ledford]{Jeff Ledford}

\address{Bushra Basit, Department of Algebra and Geometry, Budapest University of Technology and Economics,
M\H uegyetem rkp. 3, H-1111 Budapest, Hungary}
\email{bushrabasit18@gmail.com}
\address{Steven Hoehner, Department of Mathematics \& Computer Science, Longwood University,
201 High Street, Farmville, VA 23909, U.S.A.}
\email{hoehnersd@longwood.edu}
\address{Zsolt L\'angi, Department of Algebra and Geometry, Budapest University of Technology and Economics,
M\H uegyetem rkp. 3, H-1111 Budapest, Hungary, and HUN-REN Alfr\'ed R\'enyi Institute of Mathematics,
Re\'altanoda utca 13-15, H-1053, Budapest, Hungary}
\email{zlangi@math.bme.hu}
\address{Jeff Ledford, Department of Mathematics \& Computer Science, Longwood University,
201 High Street, Farmville, VA 23909, U.S.A.}
\email{ledfordjp@longwood.edu}

\thanks{Partially supported by the ERC Advanced Grant ``ERMiD'', and the NKFIH grant K147544.}

\begin{document}

\setcounter{footnote}{0}

\begin{abstract}\noindent
The aim of this paper is to introduce a generalization of Steiner symmetrization in Euclidean space for spherical space, which is the dual of the Steiner symmetrization in hyperbolic space introduced by J. Schneider (Manuscripta Math. 60: 437–461 (1988)). We show that this symmetrization preserves volume in every dimension, and convexity in the spherical plane, but not in dimensions $n > 2$. In addition, we investigate the monotonicity properties of the perimeter and diameter of a set under this process, and find conditions under which the image of a spherically convex disk under a suitable sequence of Steiner symmetrizations converges to a spherical cap. We apply our results to prove a spherical analogue of a theorem of Sas, and to confirm a conjecture of Besau and Werner (Adv. Math. 301: 867-901, 2016) for centrally symmetric spherically convex disks. Lastly, we prove a spherical variant of a theorem of Winternitz.

\end{abstract}

\renewcommand{\thefootnote}{}

\renewcommand{\thefootnote}{\arabic{footnote}}
\setcounter{footnote}{0}

\maketitle

\section{Introduction}

Steiner symmetrization of convex bodies in the $n$-dimensional Euclidean space $\R^n$ is a powerful tool to solve geometric optimization problems. It has many desirable properties which lend to its usefulness. In particular, it preserves volume and is nonincreasing for the other quermassintegrals as well as the diameter, the symmetrization of a convex body is a convex body, and for any convex body there exists a sequence of Steiner symmetrizations that converges to a Euclidean ball. These properties allow one to use the Steiner symmetrization process to prove many fundamental results in convex geometry, including the isoperimetric, isodiametric and Brunn--Minkowski inequalities. For background on the geometric aspects of Steiner symmetrizations of convex bodies, we refer the reader to, for example, the books of Gruber \cite{GruberBook} and Schneider \cite{SchneiderBook}.

Due to its usefulness, there have been several attempts to generalize the concept of Steiner symmetrization to other spaces of constant curvature. In particular,
in 1985 B\"or\"oczky \cite{Boroczky-1985} introduced a variant of Steiner symmetrization in the spherical space $\Sp$, which increases volume but does not preserve convexity. In 1988, J. Schneider \cite{Schneider-1988}  defined a different variant of this concept in the hyperbolic space $\HH^n$ which preserves volume, but the Steiner symmetral of a convex body is in general not convex. These symmetrizations generalize the classical Euclidean concept in different ways. In 2008, 
Leichtweiss \cite{LW2008} in a short paper generalized Schneider's symmetrization to both the spherical and hyperbolic planes, and found a certain natural condition under which it preserves convexity. For more information on the symmetrizations in \cite{Boroczky-1985, Peyerimhoff}, the interested reader is referred to \cite{Basit-Langi-2022}. Finally, we mention a recent manuscript of Lin and Wu \cite{LW2024} in which they use a rescaled Euclidean Steiner symmetrization in a certain model space of the spherical or hyperbolic space.

In this paper, we introduce and study a variant of Steiner symmetrization on $\Sp$. Our variant coincides with that of Leichtweiss \cite{LW2008} for the planar case $n=2$. We show that it preserves volume, and on $\Sph^2$, on the family of spherically convex disks, Steiner symmetrizations satisfying certain additional properties preserve convexity. We investigate the properties of this symmetrization process, and as an application, we prove 
a theorem of Sas \cite{Sas-1939} about the area of polygons inscribed in a convex disk, and a conjecture of Besau and Werner \cite{Besau-Werner} about the floating area of convex bodies in $\Sp$, for centrally symmetric convex disks. Finally, we prove a spherical version of the classical Winternitz theorem.

The structure of the paper is as follows. In Section~\ref{sec:Steiner}, we introduce our new symmetrization process, and we investigate its properties. More specifically, in Subsection~\ref{subsec:prelim} we give a brief overview of the properties of spherical geometry that we are going to use in the paper, and we define our symmetrization on $\Sp$. In Subsection~\ref{subsec:volume} we show that spherical volume does not change under symmetrization in $\Sp$. In Subsection~\ref{subsec:convexity}, we prove that if $K$ is a spherically convex disk in $\Sph^2$ satisfying the so-called \emph{angular monotonicity property} with respect to the symmetrization, then its image is spherically convex.  We also show that this result does not hold in $\Sp$ for any $n > 2$. In Subsections~\ref{subsec:perimeter} and \ref{subsec:diameter}, we show that the perimeter and the diameter of a convex disk in $\Sph^2$ do not increase under a Steiner symmetrization. Finally, in Subsection~\ref{subsec:convergence} we prove that if $K$ is a centrally symmetric convex disk in $\Sph^2$, or its diameter is less than $\frac{\pi}{2}$, then applying subsequent Steiner symmetrizations on $K$ a spherical cap of the same area can be approximated arbitrarily well in the Hausdorff metric.

In Section~\ref{sec:appl} we deal with applications of our symmetrization process. In particular, in Subsection~\ref{subsec:Sas} we prove that for any $N \geq 3$ and $0< A < 2 \pi$, among centrally symmetric convex disks $K$ of area $A$, the area of a maximum area convex $N$-gon contained in $K$ is minimal if $K$ is a spherical cap. This result is a variant of a theorem of Sas \cite{Sas-1939}, proving the same property of Euclidean circular disks in $\R^2$ in the family of convex disks. In addition, in Subsection~\ref{subsec:floating} we verify a conjecture of Besau and Werner \cite{Besau-Werner} for centrally symmetric convex disks, about an isoperimetric property of spherical balls regarding their floating areas.

The final section contains a spherical variant of a theorem of Winternitz about the area ratio of two regions obtained by intersecting a spherically convex disk through its centroid. The proof of this statement, even though it does not use Steiner symmetrization directly, relies on the tools used in Section~\ref{sec:Steiner}.


\section{Steiner symmetrization on the sphere}\label{sec:Steiner}

\subsection{Preliminaries}\label{subsec:prelim}

This paper investigates problems in the $n$-dimensional spherical space $\mathbb{S}^n$, which we regard as the unit sphere in the $(n+1)$-dimensional Euclidean space $\R^{n+1}$, centered at the origin $o$. More specifically, if $\langle \cdot, \cdot \rangle $ denotes the standard inner product of $\R^{n+1}$ and $\| x \| = \sqrt{\langle x,x\rangle}$ denotes the Euclidean norm of $x \in \R^{n+1}$, then we set $\Sp = \{ x \in \R^{n+1} : \| x \| = 1 \}$.
The \emph{spherical distance} of points $x,y \in \Sp$ is defined as
\[
d_s(x,y) = \arccos \left( \langle x,y\rangle \right).
\]
Here and throughout the paper, unless we state  otherwise, by ``distance" we always mean ``spherical distance".
Points of $\Sp$ at distance $\pi$ are called \emph{antipodal}. A pair of points is antipodal if it is the intersection of $\Sp$ and a line in $\R^{n+1}$ through $o$. More generally, the intersection of $\Sp$ with a $(k+1)$-dimensional linear subspace of $\R^{n+1}$, where $k \geq 0$, is called a \emph{$k$-dimensional subspace} or \emph{great sphere} of $\Sp$. We call $1$-dimensional great spheres \emph{great circles}, or \emph{(spherical) lines}. For any distinct and nonantipodal points $x,y \in \Sph^n$, there is a unique line containing them, which they decompose into two closed noncongruent arcs; the shorter such arc, denoted by $[x,y]_s$, is called the \emph{segment} with endpoints $x,y$. 
For any $0 < r < \pi$ and $x \in \Sp$, the set $B_s(x,r)=\{y\in\Sp:\, d_s(x,y)\leq r\}$ is called the \emph{closed (spherical) ball} of radius $r$ and center $x$. The (relative) interior of a closed ball is called an \emph{open ball}. Closed/open balls with $r=\pi/2$ are called closed/open \emph{hemispheres}.  
The (relative) boundary of a set $S\subset\mathbb{S}^n$ is denoted $\bd(S)$.

A set $K\subset\Sp$ is called  \emph{(spherically) convex} if it is contained in an open hemisphere, and for every $x,y\in K$, the segment $[x,y]_s$ is contained in $K$. The intersection of convex sets in $\Sp$ is convex. In particular, for any set $A\subset\Sp$ contained in an open hemisphere, the intersection of all convex sets containing $A$ is convex; this set is called the \emph{(spherical) convex hull} of $A$, denoted by $\conv_s (A)$. 
If $A \subset \Sp$ is a finite set contained in an open hemisphere, the set $P=\conv_s(A)$ is called a \emph{(spherical) polytope}. In this case, if $A$ satisfies the condition that $\conv_s (B) \neq P$ for any $B \subsetneq A$, then $A$ is called a \emph{minimal representation} of $P$. Every polytope has a unique minimal representation, and its elements are called the \emph{vertices} of $P$.

A compact, spherically convex set with nonempty interior is called a \emph{(spherical) convex body}. The $n$-dimensional Hausdorff measure  of a spherical convex body $K$ is called its \emph{(spherical) volume}, denoted by $\vol_s(K)$. In the case $n=2$, we write $\area(K)$ for $\vol_s(K)$.

In the following, we fix some open hemisphere $S$ of $\Sp$ with center $c$, and denote the family of all convex bodies of $\Sp$ contained in $S$ by $\K_S$.
Let $\bar{L}$ and $\bar{H}$ be a great circle and an $(n-1)$-dimensional great sphere, respectively, meeting orthogonally at $c$, and set $L=\bar{L} \cap S$, $H = \bar{H} \cap S$. For any $2$-dimensional great sphere $\bar{G}$ through $\bar{L}$ and $0 < \delta < \frac{\pi}{2}$, the set of points of $G$ at spherical distance $\delta > 0$ from $\bar{L}$ consists of two circles $\bar{C}_1, \bar{C}_2$ of spherical radius $\frac{\pi}{2}-\delta$, each concentric with one of the closed hemispheres in $\bar{G}$ bounded by $\bar{L}$. We call the curves $C_i =\bar{C}_i \cap S$, $i=1,2$, \emph{distance curves} with axis $L$ and distance $\delta$. We regard $L$ as the only distance curve with axis $L$ and distance $0$, and remark that any point $x \in S$ belongs to a unique distance curve with axis $L$. For later use, we denote the set of points of $\bar{H}$ at spherical distance $\frac{\pi}{2}$ from $L$ by $H_0$; this set is an $(n-2)$-dimensional great sphere of $\Sp$, contained in $\bd (S)$.

Our main definition is the following.

\begin{definition}\label{defn:Steiner}
Let $X \subset S$ be compact. Assume that for any distance curve $C$ with axis $L$, the set $C \cap X$ is connected. Then the \emph{Steiner symmetral} of $X$ with respect to $(L,H)$, denoted by $\sigma_{L,H}(X)$, is defined as follows. For any distance curve $C$ with axis $L$:
\begin{enumerate}
\item[(i)] if $X \cap C = \emptyset$, then $\sigma_{L,H}(X) \cap C=\emptyset$;
\item[(ii)] if $X \cap C$ is a singleton, then  $\sigma_{L,H}(X) \cap C = C \cap H$;
\item[(iii)] if $X \cap C$ is a closed circular arc, then $\sigma_{L,H}(X) \cap C$ is the unique closed circular arc in $S$ with the same length, centered at $C \cap H$.
\end{enumerate}
\end{definition}

Motivated by Definition~\ref{defn:Steiner}, from now on, if a set $X \subset S$ satisfies the property that for any distance curve $C$ with axis $L$ the set $X \cap C$ is connected, then we say that $X$ satisfies the \emph{connectedness property}.

\begin{remark}\label{rem:basicproperties}
By construction, $\sigma_{L,H}(K)$ is symmetric with respect to $H$ and is invariant on $H$-symmetric sets in $S$. It is also monotonic, i.e., if the compact sets $K \subseteq L$ satisfy the connectedness property, then $\sigma_{L,H}(K) \subseteq \sigma_{L,H}(L)$. Furthermore, Steiner symmetrization is idempotent, meaning that $\sigma_{L,H}(\sigma_{L,H}(K))=\sigma_{L,H}(K)$ holds for any compact set $K$ with the connectedness property. 
\end{remark}

\begin{remark}\label{rem:validity}
We note that any spherical ball satisfies the connectedness property. On the other hand, this is not true in general for every spherically convex body in $S$. To provide such an example for $n=2$, we may take a spherical polygon not intersecting $L$ which has more than one vertex closest to $L$. For $n > 2$, a straightforward modification of this example shows this observation.
\end{remark}

\begin{remark}\label{rem:validity2}
An elementary observation shows that if $K \subset S$ is a convex body with $c \in K$, or more generally, if $K \cap L \neq \emptyset$, then $K$ satisfies the connectedness property. Indeed, if $q_1,q_2 \in K \cap C$, then by convexity, the spherical triangle with vertices $q_1,q_2,c$ contains the arc of $C$ between $q_1$ and $q_2$, and thus the set $C \cap K$ is connected.
\end{remark}

\begin{remark}
    Several months after our paper was submitted, Lin and Deng submitted the  paper \cite{Lin-Deng-2024} in which they also defined the Steiner symmetral of a set following the same idea as in Definition~\ref{defn:Steiner} of our paper (without giving a reference to it). They derived some of its elementary properties, most of which are already investigated in this paper.
\end{remark}

\begin{definition}\label{defn:projection}
For any set $X \subset S$, the \emph{projection of $X$ onto $H$ along $L$} is the set of points $q \in H$ with the property that the distance curve through $q$ with axis $L$ intersects $X$. We denote this set by $\proj_{L,H}(X)$.
\end{definition}

In all parts of the paper, apart from Subsection~\ref{subsec:volume} and Subsubsection~\ref{subsubsec:highdimconv}, we deal with the spherical plane $\Sph^2$. In all of these parts, we use the following coordinate system on $\R^3$: $c=(1,0,0)$, $L$ lies in the $(x,y)$-plane, and $H$ lies in the $(x,z)$-plane. We use the normal polar coordinates for any point $p \in S$, namely $\theta_p$ denotes the (signed) spherical distance of $p$ from $L$, and $\varphi_p$ denotes the (signed) spherical distance of the orthogonal projection of $p$ onto $L$ from $H$, implying also that $-\frac{\pi}{2} < \theta_p, \varphi_p < \frac{\pi}{2}$. We call the points $(0,0,1)$ and $(0,0,-1)$ the \emph{poles} of $\Sph^2$, and for any $-\frac{\pi}{2} < \theta < \frac{\pi}{2}$, we denote the distance curve $\{ p : \theta_p = \theta \}$ by $C_{\theta}$.

\subsection{Volume of the Steiner symmetral of a set}\label{subsec:volume}

\begin{theorem}\label{thm:volume}
For any compact set $X \subset S$ with the connectedness property, $\vol_s(X)=\vol_s(\sigma_{L,H}(X))$.
\end{theorem}

\begin{proof}
Let $L_1$ denote the linear hull of $\bar{L}$, and let $L_2$ be the orthogonal complement of $L_1$ in $\R^{n+1}$. Note that  $H_0=L_2 \cap \Sp$ is the set of points of $\bar{H}$ at spherical distance $\frac{\pi}{2}$ from $\bar{L}$, and thus, it is a subset of $\bd(S)$.
Clearly, any point $p \in \R^{n+1}$ can be uniquely written as $p=p_1+p_2$ with $p_i \in L_i$ for $i=1,2$, implying that any point $q \in \Sp$ can be written as $q = (\cos \theta) q_1 + (\sin \theta) q_2$, where $q_1 \in \bar{L}$, $q_2 \in H_0$, and $0 \leq \theta \leq \frac{\pi}{2}$. Furthermore, this representation is unique if $\theta \neq 0, \frac{\pi}{2}$, and the property that $q \in S$ is equivalent to the condition that $q_1 \in L$.

In the following, we use a Descartes coordinate system of $\R^{n+1}$ in which the plane spanned by the first two basis vectors is $L_1$, and the $(n-1)$-dimensional linear subspace spanned by the remaining basis vectors is $L_2$. 
We write each point $x=(x_1,\ldots,x_{n+1}) \in \R^{n+1}$ as $x= (r\cos \theta\cos \phi, r\cos \theta\sin\phi, (r\sin\theta)q_2)$, where $r \geq 0$, $0 \leq \theta \leq \frac{\pi}{2}$, $-\pi \leq \phi \leq \pi$, and $q_2=(q_{2,1},\ldots,q_{2,n-1})$. The Jacobian determinant $\det(J)$ of this coordinate transformation on $\R^{n+1}$ can be expressed as
\begin{align*}
    \det (J) &= \det\left(\frac{\partial(x_1,\ldots,x_{n+1})}{\partial(r,\theta,\phi,q_{2,1},\ldots,q_{2,n-1})}\right)\\
   & =\det \left( \begin{array}{ccc|ccc}
    \cos\theta\cos\phi & -r\sin\theta\cos\phi & -r\cos\theta\sin\phi & 0  &\cdots & 0 \\
    \cos\theta\sin\phi & -r\sin\theta \sin\phi & r\cos\theta\cos\phi & 0 & \cdots & 0 \\ \hline
    (\sin\theta)q_{2,1}& (r\cos \theta)q_{2,1} & 0 &&& \\
    \vdots & \vdots & \vdots &   & (r\sin\theta)D(q_2) &  \\
    (\sin\theta)q_{2,n-2}& (r\cos \theta)q_{2,n-2} & 0 &&&
\end{array} \right),
\end{align*}
where $D(q_2)$ is the $(n-1) \times (n-2)$ derivative matrix of $q_{2}$. From this we obtain 
\begin{align*}
      \det(J)& =(r^{n}\sin^{n-2}\theta)\det \left( \begin{array}{ccc|ccc}
    \cos\theta\cos\phi & -\sin\theta\cos\phi & -\cos\theta\sin\phi & 0  &\cdots & 0 \\
    \cos\theta\sin\phi & -\sin\theta \sin\phi & \cos\theta\cos\phi & 0 & \cdots & 0 \\ \hline
    (\sin\theta)q_{2,1}& (\cos \theta)q_{2,1} & 0 &&& \\
    \vdots & \vdots & \vdots &   & D(q_2) &  \\
    (\sin\theta)q_{2,n-1}& (\cos \theta)q_{2,n-1} & 0 &&&
\end{array} \right)\\
&=-(r^{n}\sin^{n-2}\theta\cos\theta\sin\phi)\det \left( \begin{array}{cc|ccc}
    \cos\theta\sin\phi & -\sin\theta \sin\phi  & 0 & \cdots & 0 \\ \hline
    (\sin\theta)q_{2,1}& (\cos \theta)q_{2,1}  &&& \\
    \vdots & \vdots  &   & D(q_2) &  \\
    (\sin\theta)q_{2,n-1}& (\cos \theta)q_{2,n-1}  &&&
\end{array} \right)\\
& -(r^{n}\sin^{n-2}\theta\cos\theta\cos\phi)\det \left( \begin{array}{cc|ccc}
    \cos\theta\cos\phi & -\sin\theta \cos\phi  & 0 & \cdots & 0 \\ \hline
    (\sin\theta)q_{2,1}& (\cos \theta)q_{2,1}  &&& \\
    \vdots & \vdots  &   & D(q_2) &  \\
    (\sin\theta)q_{2,n-2}& (\cos \theta)q_{2,n-2}  &&&
\end{array} \right)\\
&=-(r^{n}\sin^{n-2}\theta\cos\theta\sin\phi)\bigg((\cos\theta\sin\phi)\det\left((\cos\theta)q_2^T;D(q_2)\right)\\
&\qquad\qquad\qquad\qquad\qquad\qquad+(\sin\theta\sin\phi)\det\left((\sin\theta)q_2^T;D(q_2)\right)\bigg)\\
&\phantom{=}-(r^{n}\sin^{n-2}\theta\cos\theta\cos\phi)\bigg((\cos\theta\cos\phi)\det\left((\cos\theta)q_2^T;D(q_2)\right)\\
&\qquad\qquad\qquad\qquad\qquad\qquad+(\sin\theta\cos\phi)\det\left((\sin\theta)q_2^T;D(q_2)\right)\bigg)\\
&=-(r^{n}\sin^{n-2}\theta\cos\theta)\det\left(q_2^T;D(q_2)\right).
\end{align*}

Here we remark that $\det\left(q_2^T;D(q_2)\right)dq_2$ is the (spherical) volume element of $H_0$, identified with $\mathbb{S}^{n-3}$. Thus, by Fubini's theorem, the spherical volume of $X \subset S$ is
\begin{equation}\label{eq:volume}
\vol_s(X) = \int_{0}^{\pi/2} \int_{q_{2} \in H_0} l(X \cap C(\theta, q_2)) \, \sin ^{n-2} \theta \cos \theta \, d q_2 \, d \theta,
\end{equation}
where $l(\cdot)$ denotes spherical arclength, and $C(\theta,q_2)$ is the distance curve through $q_2$ with axis $L$ and distance $\theta$;
this curve is the semicircle 
\[
(\cos\theta\cos\varphi, \cos\theta\sin\varphi, (\sin \theta)q_2), \quad - \frac{\pi}{2} < \varphi < \frac{\pi}{2}.
\]
Note that the expression in (\ref{eq:volume}) does not change under the Steiner symmetrization $\sigma_{L,H}$, implying the assertion.
\end{proof}

\begin{remark}
The planar version of this result may be found in \cite{LW2008}.     
\end{remark}

\subsection{Convexity of the Steiner symmetral}\label{subsec:convexity}

\subsubsection{Convexity on $\Sph^2$}\label{subsubsec:planarconv}

In this subsection, we use the coordinate system introduced in the last paragraph of Subsection~\ref{subsec:prelim}.

\begin{definition}\label{defn:anglength}
Let $C \subset S$ be an arc of a distance curve with axis $L$. Let $p,q$ denote the endpoints of $C$. Then the quantity $|\varphi_p-\varphi_q|$ is called the \emph{angular length} of $C$.
\end{definition}

\begin{definition}\label{defn:monotonicity}
Let $K \in \K_S$ satisfy the connectedness property. Assume that the angular length of $K \cap C_{\theta}$ is nonincreasing for $\theta \geq 0$, and nondecreasing for $\theta \leq 0$. Then we say that $K$ satisfies the \emph{angular monotonicity property}.
\end{definition}

\begin{lemma}\label{lem:monotonicity}
Let $K \in \K_S$ satisfy the connectedness property. Then $K$ satisfies the angular monotonicity property if and only if $K$ intersects $L$ in a spherical segment $[q_1,q_2]_s$, and $K$ has a pair of supporting lines $L_1,L_2$ at $q_1,q_2$, respectively, such that $L_1,L_2$ are symmetric to the midpoint of $[q_1,q_2]_s$.
\end{lemma}

By a \emph{lune}, we mean  the intersection of two closed hemispheres of $\Sp$, bounded by different great spheres. Clearly, no lune is contained in $S$ as they contain antipodal points. Thus, with a little abuse of terminology, in the remaining part of this section we call a lune the part, in $S$, of the intersection of two closed hemispheres. Before proving Lemma~\ref{lem:monotonicity}, we make the following observation.

\begin{remark}\label{rem:equivalence}
The condition of the lemma for $L_1,L_2$ is equivalent to the property that the midpoint of $[q_1,q_2]_s$ is the center of the lune in $S$ bounded by $L_1$ and $L_2$, and also to the property that if $z$ is any of the two intersection points of $L_1$ and $L_2$, then the sum of the angles of the triangle $\conv_s \{ q_1,q_2, z \}$ at $q_1$ and $q_2$ is equal to $\pi$.
\end{remark}

\begin{proof}[Proof of Lemma~\ref{lem:monotonicity}]
We only prove the `if' part of the statement, as the opposite direction can be shown by a straightforward modification of our argument.
 
Assume that $K$ intersects $L$ in a spherical segment $[q_1,q_2]_s$, and that it has a pair of supporting lines $L_1, L_2$ satisfying the condition in the lemma. Let us denote by $K'$ the lune (in $S$) bounded by $L_1$ and $L_2$. Let $h$ denote the rotation around the poles satisfying $h(q_1)=q_2$. Then by our conditions, $h(L_1)=L_2$. This shows that the angular length of an arc $C \cap K'$ is the same for any distance curve $C$ with axis $L$ that intersects both $L_1$ and $L_2$ in $S$. In particular, this implies Lemma~\ref{lem:monotonicity} for the special case that $K=K'$.

Now let $q_1', q_2' \in [q_1,q_2]_s$ be such that $q_1'$ is closer to $q_1$ than $q_2'$. Let $L_1'$ and $L_2'$ be the rotated copies of $L_1$ by the rotation around the poles that maps $q_1$ into $q_1'$ and $q_2'$, respectively. By the definition of spherical convexity, for $i=1,2$, $L_i'$ decomposes $\bd(K)$ into two connected arcs $\Gamma_1$ and $\Gamma_2$, exactly one of them containing $q_1$. We denote this arc by $\Gamma_1$. Clearly, by our conditions, $\Gamma_1 \subseteq \Gamma_2$.

On the other hand, the set of points $p \in \bd(K)$ with $\theta_p \geq \theta $ is also a connected arc of $\bd(K)$, which we denote by $\Psi_{\theta}$. Observe that the endpoints of $\Psi_{\theta}$ are the endpoints of the arc $C_{\theta} \cap K$. This shows that if $\theta_2 > \theta_1$, then $\Psi_{\theta_2} \subseteq \Psi_{\theta_1}$. Combining this with the observation in the previous paragraph, we obtain that if $\theta_2 > \theta_1 \geq 0$, then the angular length of $C_{\theta_1} \cap K$ is not less than that of $C_{\theta_2} \cap K$. If $0 \geq \theta_2 > \theta_1$, a similar observation yields the assertion.
\end{proof}

\begin{remark}\label{rem:centralsymmetry}
By Lemma~\ref{lem:monotonicity}, if $K \in \K_S$ is symmetric about  $c$, then $K$ clearly satisfies the angular monotonicity property.
\end{remark}

Before our next theorem, we remark that in a slightly different form, stated for curves with $C^{\infty}$-class boundaries, Theorem~\ref{thm:planarconvexity} appeared as the main result, Theorem 4.3, of \cite{LW2008}. The proof of that theorem is based on computing the geodesic curvatures at the boundary points of the symmetral $\sigma_{L,H}(K)$ of the convex disk $K$. Here we present a proof that uses only elementary geometry.

\begin{theorem}\label{thm:planarconvexity}
Assume that $K \in \K_S$ satisfies the angular monotonicity property. Then $\sigma_{L,H}(K)$ is convex.
\end{theorem}

\begin{proof}
First, note that by Definition~\ref{defn:anglength}, $K$ satisfies the connectedness property. In the proof, for any $\theta_1 < \theta_2$, we call the connected region of $S$ bounded by the distance curves $C_{\theta_1}$ and $C_{\theta_2}$ a \emph{horizontal strip} of $S$. Note that by a standard continuity argument, it is sufficient to prove the statement for the case that $K$ is a convex polygon.

Suppose for contradiction that there is a polygon $P$ such that $\sigma_{L,H}(P)$ is not convex. By Lemma~\ref{lem:monotonicity}, we may assume that $P$ satisfies the angular monotonicity property.
For any distance curve $C_{\theta}$, let us denote the endpoints of $C_{\theta} \cap \sigma_{L,H}(P)$ by $p^-(\theta)$ and $p^+(\theta)$ such that the $\varphi$-coordinate of $p^-(\theta)$ is less than or equal to that of $p^+(\theta)$.

Note that if $\sigma_{L,H}(P)$ is not convex, then there is a horizontal strip $T$, bounded by the distance curves $C_{\theta_1}$ and $C_{\theta_2}$, where $\theta_2 > \theta_1$ can be assumed to be nonnegative and are sufficiently close, such that:
\begin{enumerate}
\item[(i)] $T \cap \bd(P)$ consists of two segments, and
\item[(ii)] we have $\area(T \cap P) < \area(Q)$, where $Q$ is the connected region bounded by $\sigma_{L,H}(C_{\theta_1} \cap P)$, $\sigma_{L,H}(C_{\theta_2} \cap P)$, and the two segments $[p^-(\theta_1), p^-(\theta_2)]_s$ and $[p^+(\theta_1),p^+(\theta_2)]_s$.
\end{enumerate}

For $i=1,2$, let $\varphi_i$ denote the angular length of $C_{\theta_i} \cap P$. Then by the angular monotonicity property, $\varphi_2 \leq \varphi_1$. If $\varphi_2 = \varphi_1$, then the segments in $H \cap \bd(P)$ are rotated copies of each other by $\varphi_2$ around the poles. Since Steiner symmetrization does not change the angular length of $C_{\theta} \cap P$ for any value of $\theta$, in this case $Q$ coincides with $\sigma_{L,H}(T \cap P)$, and Theorem~\ref{thm:volume} yields the assertion.

\begin{figure}[ht]
\begin{center}
\includegraphics[width=0.8\textwidth]{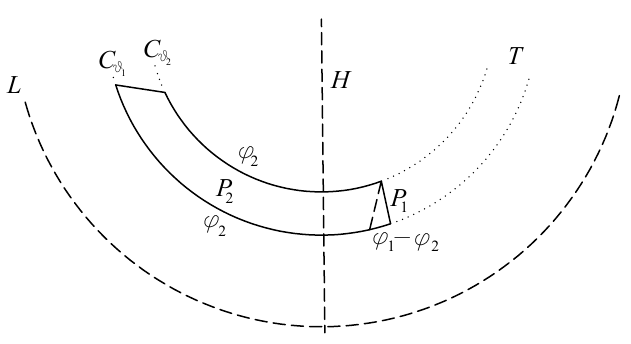}
\caption{An illustration for the dissection of $P$ in the proof of Theorem~\ref{thm:planarconvexity}. The continuous line in the figure indicates $\bd (P)$. The dashed segment in $P$, which is a rotated copy of the left segment in $\bd(P)$ in the figure around the poles, decomposes $P$ into the regions $P_1, P_2$. The symbols $\varphi_2, \varphi_1-\varphi_2$ show the angular lengths of the corresponding distance curve arcs in the boundaries of $P_1$ and $P_2$.}
\label{fig:dissection}
\end{center}
\end{figure}

Assume that $\varphi_2 < \varphi_1$. Then the rotated copy of one component of $T \cap \bd(P)$ by $\varphi_2$ towards the other component decomposes $T \cap P$ into 
two components $P_1$ and $P_2$ such that $P_1$ contains exactly one point of $C_{\theta_2} \cap P$ and an arc of angular length $\varphi_1-\varphi_2$ of $C_{\theta_1} \cap P$, and $P_2$ contains two arcs of angular length $\varphi_2$ of $C_{\theta_i} \cap P$ for $i=1,2$ (see Figure~\ref{fig:dissection}). Note that we can obtain the Steiner symmetral of $T \cap P$ by taking the Steiner symmetral of $P_2$, and rotating both curves in its boundary, connecting the intersection point of $H$ and $C_{\theta_2}$ to a point of $C_{\theta_1}$, by $\frac{\varphi_2}{2}$ around the poles and away from $H$. In other words, we may assume that $\varphi_2=0$, implying that $p^-(\theta_2)= p^-(\theta_2)$ is the intersection point of $H$ and $C_{\theta_2}$. Let us denote this common point by $p(\theta_2)$. Now, recall from spherical geometry that for any fixed segment $[x,y]_s$, the points $z$ with the property that the area of the triangle $\conv_s \{x,y,z\}$ is fixed, moves on a circle containing $-x,-y$. This circle is called a \emph{Lexell circle} of $[x,y]_s$. Observe that the Lexell circle of $[p^-(\theta_1),p^+(\theta_1)]_s$ intersecting $C_{\theta_2}$, belonging to smallest area triangles, passes through $p_{\theta_2}$. This contradicts the assumption that $\area(T \cap P) < \area(Q)$.
\end{proof}

\subsubsection{Convexity on $\Sp$ with $n > 2$}\label{subsubsec:highdimconv}

In this subsubsection, we investigate the convexity properties of the set $\sigma_{L,H}(K)$ where $K$ is spherically convex. During this, we often use the central projection of $S$ from $o$ onto the tangent hyperplane $T_c \Sp$ of $\Sp$ at the center $c$ of $S$. We denote this projection by $\pi_c : S \to T_c \Sp$. Recall that for any $X \subseteq S$, $\pi_c(X)$ is convex if and only if $X$ is spherically convex. To simplify our notation, for any $p \in S$, we set $p^*=\pi_c(p)$, and we use the same notation for subsets of $S$.

\begin{lemma}\label{lem:model}
Consider the distance curve $C_{\theta}$ of $S \subset \Sph^2$, where $- \frac{\pi}{2} < \theta < \frac{\pi}{2}$. If $\theta=0$, i.e., if $C_{\theta}=L$, then $\pi_c (C_{\theta})$ is the straight line
\[
z=0
\]
in the tangent plane $T_c \mathbb{S}^2= \{ x=1 \}$.
If $\theta > 0$, then $\pi_c (C_{\theta})$ is the branch of the hyperbola
\[
1 = \frac{z^2}{\tan^2 \theta} - y^2
\]
in $\{ x=1 \}$ with $z > 0$. If $\theta < 0$, then $\pi_c (C_{\theta})$ is the branch of the same hyperbola with $z < 0$.
\end{lemma}

\begin{proof}
Let $q$ be a point of $C_{\theta}$, where we assume that $\theta > 0$. Using spherical polar coordinates, we have $q= (\cos \theta \cos \varphi, \cos \theta \sin \varphi, \sin \theta)$ for some $- \frac{\pi}{2} < \varphi < \frac{\pi}{2}$. Since the equation of $T_c \Sph^2$ is $\{ x=1 \}$, the central projection of $q$ onto this plane is
\[
\pi_c(q) = \left( 1, \tan \varphi, \frac{\tan \theta}{ \cos \varphi} \right).
\]
Using the identity $1 + \tan^2 \alpha = \frac{1}{\cos^2 \alpha}$, we may eliminate $\varphi$ and obtain that the coordinates of $\pi_c (q)$ satisfy the equations $x=1$ and $1 = \frac{z^2}{\tan^2 \theta} - y^2$. Since $C_{\theta}$ is connected and $\pi_c$ is continuous, one can easily show that $\pi_c(C_{\theta})$ is one of the two branches of the above hyperbola. Now, since points with positive $\varphi$-coordinates on $S$ are projected onto the open half-plane $\{ x = 1 \}\cap\{ z > 0 \}$, the assertion follows for $\theta > 0$. The remaining two cases can be proved by a similar argument.
\end{proof}

\begin{lemma}\label{lem:proj}
Let $T \subset S$ be any triangle with $c$ as a vertex. Then $\proj_{L,H}(T)$ is convex.
\end{lemma}

\begin{proof}
Note that $T \cup L$ is contained in a $3$-dimensional great sphere. Since this great sphere contains every distance curve with axis $L$ through a point of $T$, we may assume that it coincides with $\Sp$, i.e., that $n=3$. For any point $q \in S$, we say that the \emph{angular distance} of $q$ from $H$ is the angle of the rotation around $H_0$ that moves $q$ to $H$.

Furthermore, observe that if $T = \conv_s \{ c, q_1,q_2 \}$ and $T'= \conv_s \{ c,q_1',q_2'\}$, where $[q_1',q_2']_s$ is a rotated copy of $[q_1,q_2]_s$ around $H_0$, then $\proj_{L,H}(T) = \proj_{L,H}(T')$. Thus, without loss of generality, we may assume that $H$ intersects $[q_1,q_2]_s$, and that $q_1$ and $q_2$ have the same angular distance from $H$. For $i=1,2$, let $C_i$ denote the distance curve with axis $L$ through $q_i$, and let $\tilde{q}_i=\pi_c(\proj_{L,H}(q_i))$.

In the remaining part of the proof, we investigate the central projections of $T$ and $\proj_{L,H}(T)$ onto $T_c \Sph^3$. We identify this tangent space with $\R^3$, and, without loss of generality, assume that $\pi_c(L)$ is the $z$-axis and $\pi_c(H)$ is the $(x,y)$-plane. Then, using a suitable coordinate system, we have that $\tilde{q}_1 = (r_1 \cos \varphi, r_1 \sin \varphi, 0)$ and $\tilde{q}_2 = (r_2 \cos \varphi, -r_2 \sin \varphi, 0)$ for some $r_1,r_2 > 0$ and $0 < \varphi < \frac{\pi}{2}$. 

From Lemma~\ref{lem:model}, it follows that the points of the central projection $C_1^*$ of the distance curve $C_1$ are of the form $(r_1(z) \cos \varphi, r_1(z) \sin \varphi, z)$, where $1= \frac{r_1(z)^2}{\tan^2 \theta_1} - z^2$ and $\theta_1$ is the spherical distance of $C_1$ from $L$. From this and the expressions for the coordinates of $\tilde{q}_1$, we obtain that $r_1 = \tan \theta_1$ and $r_1(z) = \sqrt{1+z^2} r_1$, implying that $q_1^* = (r_1 \sqrt{1+\bar{z}^2} \cos \varphi, r_1 \sqrt{1+\bar{z}^2} \sin \varphi, \bar{z})$ for some $\bar{z} \in \R$. Since the central projections of the points of $S$ having a given angular distance from $H$ lie on two planes in $\R^3$ parallel and symmetric to the $(x,y)$-plane, a similar computation shows that then $q_2^* = (r_2 \sqrt{1+\bar{z}^2} \cos \varphi, -r_2 \sqrt{1+\bar{z}^2} \sin \varphi, -\bar{z})$. Without loss of generality, we may assume that $\bar{z} \geq 0$.

Let $m^*$ (resp. $\tilde{m}$) be the point where $[q_1^*,q_2^*]$ (resp. $[\tilde{q}_1,\tilde{q}_2]$) intersects the $(x,z)$-plane.
Since $\proj_{L,H}$ is a continuous function, it suffices to prove that the spherical distance from $L$ of the distance curve through $\pi_c^{-1}(m^*)$ is not smaller than that of the curve through $\pi_c^{-1}(\tilde{m})$.

We compute $\tilde{m}$. Clearly, $\tilde{m} = t \tilde{q}_1+(1-t) \tilde{q}_2$ for some $t \in (0,1)$, and the $y$-coordinate of $\tilde{m}$ is zero. From this we obtain that $t = \frac{r_2}{r_1+r_2}$, implying that $\tilde{m} = \left( \frac{2r_1r_2}{r_1+r_2} \cos \varphi, 0,0 \right)$. Similarly, $m^* = \left( \frac{2r_1r_2 }{r_1+r_2} \sqrt{1+\bar{z}^2} \cos \varphi, 0, \frac{\bar{z} (r_2-r_1)}{r_1+r_2} \right)$.
Now, if $C_{\tilde{m}}$ denotes the central projection of the distance curve with axis $L$ through $\pi_c^{-1}(\tilde{m})$, then the point $m'$ of $C_{\tilde{m}}$ whose $z$-coordinate is equal to that of $m^*$ is $m' = \left( \frac{2r_1r_2}{r_1+r_2} \sqrt{1 + \bar{z}^2 \left( \frac{r_2-r_1}{r_1+r_2} \right)^2} \cos \varphi, 0, \frac{\bar{z} (r_2-r_1)}{r_1+r_2} \right)$.
Thus, it suffices to show that
\[
0 \leq \frac{2r_1r_2 }{r_1+r_2} \sqrt{1+\bar{z}^2} \cos \varphi - \frac{2r_1r_2}{r_1+r_2} \sqrt{1 + \bar{z}^2 \left( \frac{r_2-r_1}{r_1+r_2} \right)^2} \cos \varphi.
\]
But this inequality readily follows from the inequality $\left| \frac{r_2-r_1}{r_1+r_2} \right| < 1$ for any $r_1, r_2 > 0$.
\end{proof}

\begin{remark}\label{rem:symmetrization}
Let $\Gamma \subset S \subset \Sph^3$ be an arc of a distance curve with axis $L$ and distance $\theta$. Let the signed angular distances of the endpoints of $\Gamma$ from $H$ be $\varphi_1, \varphi_2$; here the angular distance of a point $q$ from $H$ is said to be positive if $q$ belongs to a fixed connected component of $S \setminus H$, and it is said to be negative if it is contained in the opposite component. We choose our notation in such a way that $\varphi_1 < \varphi_2$. Let $\Gamma'=\sigma_{L,H}(\Gamma)$, and note that the signed angular distances of $\Gamma'$ from $H$ are $\pm \frac{\varphi_2-\varphi_1}{2}$.

Let $\tilde{\Gamma}$ and $\tilde{\Gamma}'$ be the central projections of $\Gamma$ and $\Gamma'$ onto $\R^3$, respectively, where we choose the coordinate system of $\R^3$ such that the projection of $H$ is the $(x,y)$-plane, the projection of $L$ is the $z$-axis, and the projections of points with positive angular distance from $H$ have positive $z$-coordinate. By Lemma~\ref{lem:model}, if $z_1 < z_2$ denote the $z$-coordinates of the endpoints of $\tilde{\Gamma}$, we have $\varphi_i = \arctan z_i$ for $i=1,2$. Thus, the $z$-coordinates of the endpoints of $\tilde{\Gamma}'$ are
\[
\pm \tan \frac{\arctan z_2-\arctan z_1}{2} = \pm \frac{z_2-z_1}{\sqrt{1+z_2^2}\sqrt{1+z_1^2}+1+z_1 z_2}.
\]
\end{remark}

\begin{theorem}\label{thm:projconvex}
For any convex body $K \in \K_S$ with $c \in \proj_{L,H}(K)$, $\proj_{L,H}(K)$ is convex.
\end{theorem}

\begin{proof}
Note that if $c \in \proj_{L,H}(K)$, then $\proj_{L,H}(K) = \proj_{L,H}(\conv_s(K \cup \{c\}))$, and hence we can assume that $c \in K$. Consider any points $q_1, q_2 \in \proj_{L,H}(K)$, and let $q \in [q_1,q_2]_s$. For $i=1,2$, choose some point $x_i \in K$ which belongs to $\proj_{L,H}^{-1}(q_i)$. Applying Lemma~\ref{lem:proj} for the triangle $T=\conv_s \{ c,x_1,x_2 \}$, we have $q \in \proj_{L,H}(T) \subseteq \proj_{L,H}(K)$.
\end{proof}

We show that Theorem~\ref{thm:planarconvexity} does not hold in $\Sp$ for any $n > 2$.

\begin{theorem}\label{thm:highdimconvexity}
For any $n > 2$, there is a $c$-symmetric convex body $K \in \K_S$ for which $\sigma_{L,H}(K)$ is not convex.
\end{theorem}

\begin{proof}
First, we construct such a body in $\Sph^3$. Similarly to the proof of Lemma~\ref{lem:proj}, we construct a set $K^*$ in $\R^3$ such that its image $K=\pi_c^{-1}(K^*) \subset S$ satisfies the required conditions. With a slight abuse of terminology, we identify any set $X \subseteq S$ with its projection onto $\R^3$, or in other words, we regard $\R^3$ as a model of the spherical space $S$. We use a coordinate system in which $L$ is the $z$-axis and $H$ is the $(x,y)$-plane.
Furthermore, for any point $q=(u,v,0)$ of $H$, we denote by $C_q$ or $C(u,v)$ the distance curve with axis $L$ passing through $q$. To work in this model, we note that spherical planes passing through $H_0$ are represented by planes parallel to the $(x,y)$-plane. We call these planes `horizontal'. Thus, a rotation about $H_0$ moves a horizontal plane into another horizontal plane, whereas it moves a nonhorizontal plane into another nonhorizontal plane, which might not be parallel to the original one.

Consider a point $p_0=(x_0,0,0) \in \R^3$ with $x_0 > 0$. Recall (see Lemma~\ref{lem:model} or the proof of Lemma~\ref{lem:proj}) that the distance curve through $p_0$ with axis $L$ is the set of points $(\sqrt{t^2+1}x_0, 0, t)$, where $t$ runs over $\R$. Note also that in our model, intersections of spherical planes with $S$ are represented by Euclidean planes. We choose two planes $H_+$ and $H_-$ through $p_0$, with equations $A_+(x-x_0)+B_+ y =z$ and $A_-(x-x_0)+B_- y =z$, respectively. We call $H_+$ and $H_-$ the \emph{upward} and \emph{downward} plane, respectively, and for $* \in \{ +,- \}$, label parameters related to $H_*$ with subscript $*$. We choose specific values for the parameters in the equations later; nevertheless, we intend to guarantee that  $L$ is not disjoint from the intersection of the open halfspaces $\{ (x,y,z) \in \R^3 : A_+(x-x_0)+B_+ y > z \}$ and $\{ (x,y,z) \in \R^3 : A_-(x-x_0)+B_- y < z \}$. A simple calculation shows that this condition is equivalent to the inequality $A_+ < A_-$.

Let $V$ be a sufficiently small neighborhood of $p_0$ in $H$, and for any distance curve $C(u,v)$ with $(u,v,0) \in V$ and for any $* \in \{ +,- \}$, we denote the intersection point of $C(u,v)$ with $H_*$ by $q_*(u,v)$, and denote the arc of $C(u,v)$ between $q_-(u,v)$ and $q_+(u,v)$ by $\Gamma(u,v)$. Applying the computations in the proof of Lemma~\ref{lem:proj}, we have that
\[
q_*(u,v) = \left(\sqrt{(z_*(u,v))^2+1} u, \sqrt{(z_*(u,v))^2+1} v, z_*(u,v) \right)
\]
for some $z_*(u,v) \in \R$, where the value of this expression is a continuous function of $(u,v)$.
In the following, we compute $z_*(u,v)$. Substituting the coordinates of $q_*(u,v)$ into the equation of $H_*$, we obtain the equation
\[
A_*(\sqrt{(z_*(u,v))^2+1} u-x_0) + B_* \sqrt{(z_*(u,v))^2+1} v = z_*(u,v)
\]
for $z_*(u,v)$.
After rearranging and squaring both sides, this equation becomes a quadratic equation for $z_*(u,v)$, with solutions
\[
\frac{A_*x_0 \pm \sqrt{(A_*^2 x_0^2+1)(A_*u+B_*v)^2-(A_* u+B_* v)^4}}{(A_*u+B_*v)^2-1}
\]
in $V$, if $|A_* u + B_* v| \neq 1$.
The solutions are continuous functions of $(u,v)$; thus, the condition that $z_*(x_0,0)=0$ yields that for $* \in \{ +,- \}$,
\begin{equation}\label{eq:zstar}
z_*(u,v)= \frac{A_*x_0 - (A_*u+B_*v)\sqrt{(A_*^2 x_0^2+1)-(A_* u+B_* v)^2}}{(A_*u+B_*v)^2-1}.
\end{equation}

Let $\Gamma'(u,v) = \sigma_{L,H}(\Gamma(u,v))$. Then by Remark~\ref{rem:symmetrization}, the endpoints of $\Gamma'(u,v)$ are
\[
\left( \sqrt{z_s(u,v)^2+1} u, \sqrt{z_s(u,v)^2+1} v, \pm z_s(u,v) \right),
\]
where
\[
z_s(u,v)= \frac{z_+(u,v)-z_-(u,v)}{\sqrt{(z_+(u,v))^2+1} \sqrt{(z_-(u,v))^2+1}+1+z_+(u,v) z_-(u,v)}.
\]
We note that the value of $z_s(u,v)$ is positive if and only if $z_+(u,v) > z_-(u,v)$.

Consider the surface defined by 
\begin{equation}\label{eq:surface}
r(u,v)=\left( \sqrt{z_s(u,v)^2+1} u, \sqrt{z_s(u,v)^2+1} v, z_s(u,v) \right)
\end{equation}
and satisfying $r(x_0,0)=p_0$.
We note that if this surface is convex, then its Gaussian curvature at every point is nonnegative. Thus, to show that it is not convex, it is sufficient to find values of the parameters such that the point $p_0$ is a hyperbolic point of the surface, i.e., the Gaussian curvature at this point is negative (see \cite{DoCarmo}). 

Consider the quantities $e(u,v) = \langle N(u,v), r''_{uu}(u,v) \rangle$, $f(u,v) = \langle N(u,v), r''_{uv}(u,v) \rangle$ and  $g(u,v) = \langle N(u,v), r''_{vv}(u,v) \rangle$, where
$N(u,v) = \frac{r'_u(u,v) \times r'_v(u,v)}{| r'_u(u,v) \times r'_v(u,v)|}$.
The Gaussian curvature $K(u,v)$ of the surface at $(u,v)$ is equal to the value of $\frac{eg-f^2}{|r_u \times r_v|^2}$ at $(u,v)$. Thus, to show that $K(x_0,0)$ is negative, it suffices to show that the quantity
\[
F(u,v)= \langle r_u \times r_v, r_{uu} \rangle \cdot \langle r_u \times r_v, r_{vv} \rangle - \left( \langle r_u \times r_v, r_{uv} \rangle \right)^2
\]
is negative at $(x_0,0)$ for certain values of $A_*, B_*$ with  $* \in \{ +,- \}$.

A computation using the Maple 18.00 software yields that if $0 \leq A_+ \leq A_-$, then
\begin{multline}
F(x_0,0)=\frac{x_0^2}{64}(B_+ + B_-)(3 A_- B_- - A_- B_+ + A_+ B_- - 3 A_+ B_+)\\
(3 A_-^3 + 3 A_-^2 A_+-3 A_- A_+^2- 3 A_- B_-^2-2 A_- B_- B_+ + A_- B_+^2 - 3 A_+^3-A_+ B_-^2 + 2 A_+ B_- B_++3 A_+ B_+^2).
\end{multline}
Thus, choosing the values $A_- = 3$, $A_+=2$, $B_-=1$, and $B_+=3$, we get $F(x_0,0)=-139x_0^2$.

Now we construct $K$. Fix some sufficiently small angular distance $\varphi > 0$, and rotate the closed halfspaces with equations $ 3(x-x_0)+y  \leq z$ and 
$2(x-x_0)+3y  \geq z$ about $H_0$ by $\varphi > 0$, respectively, where we choose the directions of the rotations differently for the two halfspaces such that $p_0$ is contained in the interiors of both rotated copies. We denote the intersection of the two rotated closed halfspaces by $D$. Note that by our construction, $D$ is spherically convex, but $\sigma_{L,H}(D)$ is not. Next, we rotate $D$ about $H_0$ such that the $z$-axis intersects the rotated copy of $D$ in a closed segment symmetric to $o$, and denote the rotated copy by $D'$. Finally, we define $K$ as the intersection of $D'$ and its reflected copy about $o$. Then $K$ is an $o$-symmetric spherically convex body with the property that $\sigma_{L,H}(K)$ is not spherically convex. 

To construct a spherically convex body in dimension $n > 3$, we can extend the two lunes defining $K$ to any higher-dimensional spherical space in a natural way.
\end{proof}

\subsection{Monotonicity of the perimeter under Steiner symmetrization}\label{subsec:perimeter}

Here again we use the coordinate system introduced in the last paragraph of Subsection~\ref{subsec:prelim}. Also, for a spherically convex set $K\subset S$, let $\perim(K)$ denote the perimeter of $K$. Our main result in this subsection is the following.

\begin{theorem}\label{thm:perimeter}
Let $K$ be a spherically convex set contained in $S$ satisfying the angular monotonicity property. Then
\[
\perim(\sigma_{L,H}(K)) \leq \perim(K).
\]
\end{theorem}

For the proof of Theorem~\ref{thm:perimeter}, we need Lemma~\ref{lem:ellipse}.

\begin{lemma}\label{lem:ellipse}
Let $p,q \in \Sph^2$ be distinct, nonantipodal points. Consider the ellipse
\[
E= \{ x \in \Sph^2 : d_s(p,x)+d_s(x,q) \leq D\}
\]
with foci $p,q$, for some $0 < D < 2\pi$.
Then $E$ is convex if and only if $D < \pi$. Furthermore, if $D= \pi$, then $E$ is the closed hemisphere with the midpoint $m$ of the segment $[p,q]_s$ as center.
\end{lemma}

\begin{figure}[ht]
\begin{center}
\includegraphics[width=0.45\textwidth]{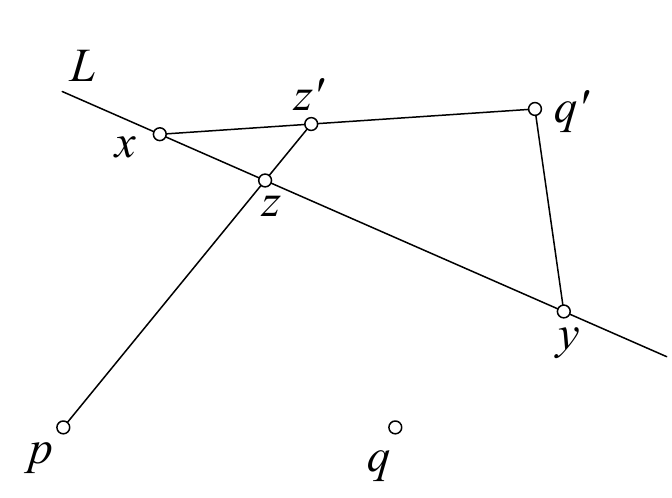}
\caption{An illustration for the proof of Lemma~\ref{lem:ellipse}.}
\label{fig:ellipse}
\end{center}
\end{figure}

\begin{proof}
First, assume that $D < \pi$, and let $\bar{E}$ denote the closed hemisphere centered at $m$.
Observe that for any $w\in\bd(\bar{E})$, we have
$d_s(w,p)+d_s(w,q)= d_s(w,-p)$ + $d_s(w,-q)$ by symmetry, yielding $d_s(w,p)+d_s(w,q) = \pi$. Thus, $E$ is contained in the interior of $\bar{E}$.

Let $x,y \in E$, and let $L$ be the great circle through $x,y$. Note that $x \in E$ implies $-x \notin E$, which yields that $x,y$ are not antipodal points. Let $z$ be any relative interior point of $[x,y]_s$. We show that $z \in E$.

Assume that $L$ does not separate $p$ and $q$. Let $q'$ denote the reflection of $q$ about the great circle $L$. Then the great circle through $[p,z]_s$ intersects $[x,q']_s$ or $[y,q']_s$ at a point $z'$. Without loss of generality, we may assume that $z' \in [x,q']_s$ (see Figure~\ref{fig:ellipse}).
Then by the triangle inequality, we have
\begin{multline*}
d_s(p,z)+d_s(z,q)=d_s(p,z)+d_s(z,q') \leq d_s(p,z)+d_s(z,z')+d_s(z',q')=d_s(p,z')+d_s(z',q') \leq \\
\leq d_s(p,x) + d_s(x,z')+d_s(z',q')=d_s(p,x)+d_s(x,q')=d_s(p,x)+d_s(x,q) \leq D.
\end{multline*}
This yields that $z \in E$. Now, if $L$ separates $p$ and $q$, then a slight modification of this argument proves the same statement, implying that $E$ is convex.

The remaining cases follow from the observation that the interior of the ellipse with foci $p,q$ and parameter $D$ is the complement of the ellipse with foci $-p,-q$ and parameter $2\pi-D$.
\end{proof}

\begin{proof}[Proof of Theorem~\ref{thm:perimeter}]
Let $\underline{\theta} , \overline{\theta}$ denote the minimum and the maximum of the $\theta$-coordinates of the points of $K$, respectively. Since the perimeter of a spherically convex set is a continuous function of the set, we can assume that the boundary of $K$ is $C^2$-class, and there are unique points of $K$ whose $\theta$-coordinates are minimal or maximal.
Then $\bd(K)$ can be written as the union of two $C^2$-class curves $\Gamma_i : [\underline{\theta}, \overline{\theta}] \to S$ ($i=1,2$), parametrized by the value of the $\theta$-coordinates of the points. We choose the indices of the curves such that for $\underline{\theta} < \theta < \overline{\theta}$, the $\varphi$-coordinate of $\Gamma_2(\theta)$ is larger than that of $\Gamma_1(\theta)$.
Thus,
\[
\perim(K)= \int_{\underline{\theta}}^{\overline{\theta}} || \dot{\Gamma}_1(\theta) || \, d \theta + \int_{\underline{\theta}}^{\overline{\theta}} || \dot{\Gamma}_2(\theta) || \, d \theta,
\]
where $|| \cdot||$ denotes the  Euclidean norm.
Similarly, $\bd(\sigma_{L,H}(K))$ can be written as the union of two $C^2$-class curves $\Delta_i : [\underline{\theta}, \overline{\theta}] \to S$ ($i=1,2$), parametrized by the value of the $\theta$-coordinates of the points such that for $\underline{\theta} < \theta < \overline{\theta}$, the $\varphi$-coordinate of $\Delta_2(\theta)$ is larger than that of $\Delta_1(\theta)$. Hence,
\[
\perim(K)= \int_{\underline{\theta}}^{\overline{\theta}} || \dot{\Delta}_1(\theta) || \, d \theta + \int_{\underline{\theta}}^{\overline{\theta}} || \dot{\Delta}_2(\theta) || \, d \theta.
\]

We intend to show that for any $\underline{\theta} \leq \theta \leq \overline{\theta}$, we have
\[
|| \dot{\Gamma}_1(\theta) || + || \dot{\Gamma}_2(\theta) || \geq || \dot{\Delta}_1(\theta) || + || \dot{\Delta}_2(\theta) ||.
\]
To do this, it suffices to show that if $0 \leq \theta \leq \theta'$ or $\theta \leq \theta' \leq 0$ are sufficiently close, then
\begin{equation}\label{eq:perimdiff}
d_s(\Gamma_1(\theta'),\Gamma_1(\theta)) + d_s(\Gamma_2(\theta'),\Gamma_2(\theta)) \geq d_s(\Delta_1(\theta'),\Delta_1(\theta)) + d_s(\Delta_2(\theta'),\Delta_2(\theta)).
\end{equation}
Without loss of generality, we may assume that $0 \leq \theta < \theta'$. For $i=1,2$, let $p_i=\Gamma_i(\theta)$, $p_i'= \Gamma_i(\theta')$, $q_i=\Delta_i(\theta)$ and $q_i'=\Delta_i(\theta')$. Note that by the definition of $\sigma_{L,H}$, we have $\varphi_{p_2}-\varphi_{p_1} = \varphi_{q_2}-\varphi_{q_1}$ and $\varphi_{p_2'}-\varphi_{p_1'} = \varphi_{q_2'}-\varphi_{q_1'}$. Furthermore, by the angular monotonicity property, the first quantity is not smaller than the second quantity.
Similarly to the proof of Theorem~\ref{thm:planarconvexity}, we may assume that $p_1=q_1$, $p_2=q_2$ and $p_1'=p_2'$, implying also that $q_1'=q_2'$, with $\varphi_{q_1'}= \varphi_{q_2'} =0$.
Since $0 \leq \theta < \theta'$ is sufficiently small, we can assume that $d_s(q_1,q_1')+d_s(q_2,q_2')=D<\pi$. Thus, by Lemma~\ref{lem:ellipse} the ellipse $E$ with foci $p_1,p_2$ and parameter $D$ is convex, and by symmetry, the distance curve $C_{\theta'}$ supports $E$ at $q_1'$. Hence, $p_1'=p_2'$ does not lie in the interior of $E$, implying (\ref{eq:perimdiff}).
\end{proof}

\subsection{Monotonicity of the diameter under Steiner symmetrization}\label{subsec:diameter}

For a spherically convex set $K\subset S$, let $\diam(K)$ denote the diameter of $K$.

\begin{theorem}\label{thm:diameter}
Let $K$ be a spherically convex set contained in $S$ and satisfying the angular monotonicity property. Then
\[
\diam(\sigma_{L,H}(K)) \leq \diam(K).
\]
\end{theorem}

\begin{proof}
Let $[p,q]_s$ be a diameter of $\sigma_{L,H}(K)$, and let $C_p$ and $C_q$ denote the distance curves with axis $L$ passing through $p$ and $q$, respectively.
Without loss of generality, we may assume that $\varphi_p \geq 0$ and $\varphi_q \leq 0$, and $|\varphi_p| \geq |\varphi_q|$.
Let $p'$ and $q'$ denote the reflections of $p,q$ about $H$, respectively. Then $p,q,p',q' \in \bd(\sigma_{L,H}K)$,  and $[p',q']_s$ is also a diameter of $\sigma_{L,H}(K)$. 

Let $\bar{p}, \bar{p}' \in \bd(K)$ denote the `preimages' of $p,p'$, respectively, under $\sigma_{L,H}$; that is, the rotation about the poles that maps $C_p \cap \sigma_{L,H}(K)$ into $C_p \cap K$, maps $p$ into $\bar{p}$ and $p'$ into $\bar{p}'$. We define the points $\bar{q}, \bar{q}' \in \bd(K)$ similarly. To prove Theorem~\ref{thm:diameter}, it suffices to show that 
\begin{equation}\label{eq:diameter}
d_s(p,q)=d_s(p',q') \leq \max \{ d_s(\bar{p},\bar{q}), d_s(\bar{p}',\bar{q}') \}.
\end{equation}

Let $y,y'$ denote the points of $C_q$ satisfying $\varphi_y=\varphi_{\bar{p}}$ and $\varphi_{y'}=\varphi_{\bar{p}'}$, and let $z,z'$ denote the points of the spherical circle $\hat{C}_q$ containing $C_q$ and opposite to $y,y'$, respectively. Note that for any point $u \in \hat{C}_q$, we have $d_s(\bar{p},y) \leq d_s(\bar{p},u) \leq d_s(p,z)$ with equality on the left if $u=y$ and on the right if $u=z$, and as we move $u$ on $\hat{C}_q$ towards $z$, then $d_s(p,u)$ strictly increases. The same statement holds if we replace $\bar{p}, y,z$ by $\bar{p}', y', z'$, respectively. Since $z,z' \notin S$, this yields that if $G \subset S$ is an arc of angular length $2|\varphi_q|$ with endpoints $u,v$, where $\varphi_u \geq \varphi_v$, then $\max \{ d_s(p,v), d_s(q,u) \}$ is minimal if the midpoints of this arc and the arc of $C_p$ with endpoints $\bar{p}, \bar{p}'$ have the same $\varphi$-coordinate. This proves (\ref{eq:diameter}).
\end{proof}

\subsection{Approximating spherical caps by subsequent Steiner symmetrizations}\label{subsec:convergence}

Gross \cite{Gross-1917} proved that for any convex body $K$ in $\R^n$, a Euclidean ball can be approximated arbitrarily well, measured in Hausdorff distance, by applying subsequent Steiner symmetrizations to $K$. Our next result is an analogue of this statement for centrally symmetric convex sets and spherical Steiner symmetrizations.

\begin{theorem}\label{thm:convergence}
Let $K \in \K_S$ be a $c$-symmetric convex disk, and let $\mathcal{F}_K$ denote the family of $c$-symmetric convex disks that can be obtained from $K$ by finitely many subsequent Steiner symmetrizations $\sigma_{L,H}$ with $L \cap H = \{ c \}$. Then there is a sequence $\{K_m\}_{m=1}^\infty$ of $c$-symmetric convex disks with $K_m \in \mathcal{F}_K$ such that $K_m\to D_c(K)$ in the Hausdorff metric as $m\to\infty$, where $D_c(K)$ is the spherical cap centered at $c$ with $\area(K)=\area(D_c(K))$.
\end{theorem}

We note that an important part of the proof of Theorem~\ref{thm:convergence} is that 
the image of any $c$-symmetric convex disk under any Steiner symmetrization $\sigma_{L,H}$ with $H \cap L = \{ c \}$ is a $c$-symmetric convex disk. 

\begin{proof}
Observe that a $c$-symmetric convex disk satisfies the angular monotonicity property for any Steiner symmetrization $\sigma_{L,H}$ with $L \cap H = \{ c \}$, implying that such a Steiner symmetrization $\sigma_{L,H}(K)$ is also a $c$-symmetric convex disk. Furthermore, Blaschke's selection theorem holds for the space of spherically convex bodies in $\Sph^2$ (see, for example, \cite[page 88]{GruberBook}). From here the proof of \cite[Theorem 9.1]{GruberBook} can be modified in a straightforward way to obtain the result.
\end{proof}

\begin{definition}\label{defn:applicableseq}
A sequence of Steiner symmetrizations $\{ \sigma_{L_i,H_i}\}_{i=1}^k$ in $\Sph^2$
is called \emph{applicable} to a convex disk $K \subset \Sph^2$ if:
\begin{itemize}
\item[(i)] $K$ is contained in the open hemisphere centered at the intersection point of $H_1$ and $L_1$, and it satisfies the angular monotonicity property with respect to $\sigma_{L_1,H_1}$; and 
\item[(ii)] for any $i=1,2,\ldots,k-1$, $\sigma_{L_{i-1},H_{i-1}}(\ldots \sigma_{L_1,H_1}(K))$ is contained in the open hemisphere centered at the intersection point of $L_i \cap H_i$, and it satisfies the angular monotonicity property with respect to $\sigma_{L_i,H_i}$.
\end{itemize}
\end{definition}

We recall that according to the notation introduced in Subsection~\ref{subsec:prelim}, $L_i$ is a half great circle and $H_i$ is an open half of a great sphere, intersecting in a unique point.

\begin{theorem}\label{thm:convergence2}
Let $K \subset \Sph^2$ be a convex disk of diameter less than $\frac{\pi}{2}$. Let $\mathcal{F}_K$ denote the family of the congruent copies of images of $K$ under finite sequences of Steiner symmetrizations applicable to $K$. There exists a sequence $\{K_m \}_{m=1}^{\infty}\subset\mathcal{F}(K)$ such that $K_m \to D(K)$ in the Hausdorff metric, where $D(K)$ is a spherical cap with $\area(D(K)) = \area(K)$.
\end{theorem}

The proof is based on the following lemma.

\begin{lemma}\label{lem:convergence2}
Let $K$ be a convex disk of diameter less than $\frac{\pi}{2}$ with circumdisk $D$ and circumradius $\rho < \frac{\pi}{2}$. Assume that some $p \in \bd(D)$ is at distance at least $\varepsilon > 0$ from $K$. Then there is some lower semicontinuous function $\delta=\delta(\rho,\varepsilon) > 0$, such that there is a finite sequence of Steiner symmetrizations $\{\sigma_{L_i, H_i}\}_{i=1}^k$ applicable to $K$ 
satisfying the property that $\sigma_{L_k,H_k}(\ldots \sigma_{L_1,H_1}(K))$ is contained in the spherical cap of radius $\rho-\delta$ and concentric with $D$.
\end{lemma}

\begin{proof}[Proof of Lemma~\ref{lem:convergence2}]
By the conditions of the lemma, there is a closed circular arc $G_0$ in $\bd(C)$ such that its distance from $K$ is at least $\varepsilon > 0$. Let $G_i \subset \bd(D)$, $i=1,2,\ldots, k$, be a family of congruent copies of $G$ which satisfies $\bigcup_{i=0}^k G_i= \bd(D)$. Let $H_i$ be the line through the center $c$ of $D$ such that the reflected copy of $G$ to $H_i$ is $G_i$. By compactness, there is a value $\delta_1>0$ such that for any line $L_1$ perpendicular to $H_1$ and $L_1 \cap H_1 \subset D$, if $\sigma_{L_1,H_1}$ is applicable to $K$, then the distance of both $G=G_0$ and $G_1$  from $\sigma_{L_1,H_1}(K)$ is at least $\delta_1$. Similarly, if $K' \subset D$ has the property that its distance from both $G_0$ and $G_1$ is at least $\delta_1$, then there is some $\delta_2 > 0$ such that for any line $L_2$ perpendicular to $H_2$ and $L_2 \cap H_2 \subset D$, if $\sigma_{L_2,H_2}$ is applicable to $K'$, then the distance of $G_0,G_1, G_2$ from $\sigma_{L_2,H_2}(K')$ is at least $\delta_2$. In general, if for any $1 \leq i \leq k-1$, some convex disk $K' \subset D$ has the property that its distance from $G_0, G_1, \ldots, G_i$ is at least $\delta_i > 0$ from $K'$, then there is some $\delta_{i+1} > 0$ such that for any line $L_{i+1}$ perpendicular to $H_{i+1}$ and $L_{i+1} \cap H_{i+1} \subset D$, if $\sigma_{L_{i+1},H_{i+1}}$ is applicable to $K'$, then the distance of $G_0,G_1, \ldots, G_{i+1}$ from $\sigma_{L_{i+1},H_{i+1}}(K')$ is at least $\delta_{i+1}$.

Thus, it remains to show that there is a sequence $\sigma_{L_i \cap H_i}$ of Steiner symetrizations applicable to $K$ such that $L_i \cap H_i \subset D$.
By continuity, without loss of generality we may assume that $K$ is smooth. Let $[p,q]_s = K \cap H_1$, and let $\Gamma_1$ and $\Gamma_2$ denote the two curves of $\bd(K)$ that connect $p,q$. Note that since $D$ is the circumdisk of $K$, the circumcenter $c$ of $K$ lies on $[p,q]_s$. For any $z \in [p,q]_s$, let $L_z$ denote the line through $z$ that is perpendicular to $H_1$. Let $z \in [p,q]_s \setminus \{p,q\}$. Then $L_z$ intersects $\Gamma_1$ and $\Gamma_2$ in unique points, which we denote by $w_1(z)$ and $w_2(z)$, respectively. Let $K(z)$ denote the closure of the component of $K \setminus L_z$ containing $q$, and for $i=1,2$, let $\alpha_i$ denote the angle of $K(z)$ at $w_i(z)$. Observe that if $z$ is close to $p$, then $\alpha_1(z)+\alpha_2(z)$ is close to $2\pi$, and if $z$ is close to $q$, then $\alpha_1(z)+\alpha_2(z)$ is close to $0$. Thus, by continuity, there is some $z \in [p,q]_s \setminus \{p,q\}$ such that $\alpha_1(z) + \alpha_2(z) = \pi$. But then $K$ is contained in the open hemisphere centered at $z$, and by Lemma~\ref{lem:monotonicity}, $K$ satisfies the angular monotonicity with respect to $\sigma_{L_z,H_1}$.
As $z \in \sigma_{L_z,H_1}(K)$ due to the construction, by continuing this process we can construct the desired sequence of Steiner symmetrizations.

Finally, the lower semicontinuity of the value $\delta(\varepsilon,\rho)$ follows from the method of construction. 
\end{proof}

Before the proof of Theorem~\ref{thm:convergence2}, we recall that every spherically convex body $K$ has a unique circumball which depends continuously on $K$ with respect to Hausdorff distance. 

\begin{proof}[Proof of Theorem~\ref{thm:convergence2}]
For any convex disk $K$, let $\rho(K)$ denote the circumradius of $K$, and let $\varepsilon(K)$ denote the Hausdorff distance between $K$ and its circumdisk.
Let $\rho = \inf \{ \rho(K') : K' \in \mathcal{F}_K\}$.
By Blaschke's selection theorem, there is a convex disk $K'$ with $\rho(K')=\varepsilon$ such that $K'$ is the limit of a sequence $\{K_m \}$ of elements of $\mathcal{F}_K$. If $K'$ is a spherical cap, then we are done. Assuming that $K'$ is not a spherical cap yields that $\varepsilon(K') > 0$. Let $\varepsilon = \varepsilon(K')$ and $0 < \delta < \delta(\rho, \varepsilon)$, where $\delta(\rho,\varepsilon)$ is the function whose existence is proved in Lemma~\ref{lem:convergence2}. Observe that $K_m \to K'$ implies that $\rho(K_m) \to \rho(K')$ and $\varepsilon(K_m) \to \varepsilon(K')$. Thus, by the lower semicontinuity of $\delta(\rho,\varepsilon)$, for all sufficiently large $m$ we have $\delta(\rho(K_m),\varepsilon(K_m)) > \delta$ and $0 \leq \rho(K_m)-\rho < \delta$. But then Lemma~\ref{lem:convergence2} yields the existence of a finite sequence $\sigma_{L_i,H_i}$ ($i=1,2,\ldots,k$) of Steiner symmetrizations applicable to $K_m$ such that 
$\sigma_{L_k,H_k}(\ldots \sigma_{L_1,H_1(K_m)})$ is contained in a spherical cap of radius
\[
\rho(K_m) - \delta(\rho(K_m),\varepsilon(K_m))  < (\rho + \delta) - \delta = \rho, 
\]
contradicting the definition of $\rho$.
\end{proof}

\section{Applications of spherical Steiner symmetrization}\label{sec:appl}

Let us recall that $S$ is an open hemisphere of $\Sph^2$ centered at $c$. In the following, we denote the family of $c$-symmetric convex disks in $S$ by $\K^{\rm sym}_S$, and the family of Steiner symmetrizations $\sigma_{L,H}$ with $L \cap H = \{ c \}$ by $\Sigma_c$ (we remark that according to our notation in Subsection~\ref{subsec:prelim}, we have $H, L \subset S$). We observe that by Remark~\ref{rem:centralsymmetry} and Theorem~\ref{thm:planarconvexity}, for any $K \in \K^{\rm sym}_S$ and $\sigma_{L,H} \in \Sigma_c$, we have $\sigma_{L,H}(K) \in \K^{\rm sym}_S$.

We start this section with a general theorem about finding optimal values of functions defined on $\K^{\rm sym}_S$.

\begin{theorem}\label{general-ineq}
  Assume that $F: \K^{\rm sym}_S \to (0,\infty)$ is nondecreasing (resp., nonincreasing) under symmetrizations $\sigma_{L,H} \in \Sigma_c$,
	and it is upper semicontinuous (resp., lower semicontinuous) with respect to the Hausdorff metric. Then for every $K \in \K^{\rm sym}_S$, we have
    \[
        F(K)\leq F(D_c(K)) \quad (\hbox{resp., } F(K)\geq F(D_c(K))),
    \]
    where $D_c(K)$ is the spherical disk centered at $c$ with $\area(D_c(K))=\area(K)$.
\end{theorem}

\begin{proof}
As in the proof of Theorem~\ref{thm:convergence}, let $\mathcal{F}_K$ denote the family of $c$-symmetric convex disks that can be obtained from $K$ by finitely many subsequent Steiner symmetrizations in $\Sigma_c$. By Theorem \ref{thm:convergence}, there exists a sequence $\{ K_m \}\subset\mathcal{F}_K$ such that $K_m \to D_c(K)$ with respect to the Hausdorff metric. Thus, if $F$ is nondecreasing under symmetrizations $\sigma_{L,H} \in \Sigma_c$, then
$F(K) \leq F(K_m)$ for any value of $m$. On the other hand, from the upper semicontinuity of $F$ it follows that
\[
\limsup_{m \to \infty} F(K_m) \leq F(D_c(K)),
\]
implying Theorem~\ref{general-ineq} in this case. If $F$ is nonincreasing under symmetrizations $\sigma_{L,H} \in \Sigma_c$ and it is lower semicontinuous, then we may apply a similar argument.
\end{proof}

Theorem~\ref{general-ineq}, combined with Theorems~\ref{thm:perimeter} and \ref{thm:diameter}, readily yields the isoperimetric and the isodiametric inequalities for $c$-symmetric convex disks. For the proofs of the general versions of these theorems, see \cite{Levy-1951,Schmidht0, Schmidt1, Schmidt2}.

\begin{corollary}\label{cor:isoperdia}
Among elements of $\K^{\rm sym}_S$ of a given area, spherical caps centered at $c$ have minimal perimeter and diameter.
\end{corollary}


\subsection{A variant of a theorem of Sas for $c$-symmetric convex disks in $\Sph^2$}\label{subsec:Sas}

The main concept of this subsection is the following.

\begin{definition}\label{defn:volumeratio}
Let $K \subset \Sp$ be a spherically convex body and let $N \geq n+1$. Then by $A_N(K)$ we denote the supremum of the volumes of all spherically convex polytopes in $K$ with at most $N$ vertices.
\end{definition}

We note that by compactness, the value $A_N(K)$ is attained at some spherically convex polytope in $K$ with at most $N$ vertices.

The minimum of this quantity has been  studied extensively in the Euclidean setting. More specifically, Blaschke \cite[pp. 49--53]{Blaschke-1923} proved that for any convex body $K$ in $\mathbb{R}^2$, there exists a triangle $T\subset K$ such that
    \begin{equation}\label{Blaschke-ineq}
\area(T) \geq \frac{3\sqrt{3}}{4\pi} \area(K),
    \end{equation}
    with equality if and only if $K$ is an ellipse. Sas \cite{Sas-1939} later extended this result to all convex polygons, showing that there exists an $N$-gon $P_N\subset K$ such that
\begin{equation}\label{Sas-ineq}
    \area(P_N) \geq \frac{N}{2\pi}\sin\frac{2\pi}{N} \area(K)
\end{equation}
with equality if and only if $K$ is an ellipse. Apart from the equality case, this result was extended by Macbeath~\cite{Macbeath} to polytopes with $N$ vertices contained in a convex body $K$ in $\R^n$. 

In our next result, we prove an analogue of the result of Sas \cite{Sas-1939} for $c$-symmetric convex disks on $\Sph^2$.

\begin{theorem}\label{spherical-macbeath-thm}
Let $K \in \K^{\rm sym}_S$ and $N \geq 3$, and let $B_c(r)$ denote the spherical cap of radius $r$ centered at $c$ that satisfies $\area(K)= \area(B_c(r))$. Then we have
\begin{equation}\label{eq:Sas}
A_N(K) \geq A_N(B_c(r)) = C(r,N),
\end{equation}
where
\begin{equation}\label{eq:regpolarea}
C(r,N)=2N\arctan\left(  \dfrac{\cos\frac{\pi}{N}}{\sin \frac{\pi}{N}\cos r} \right)-(N-2)\pi.    
\end{equation}
\end{theorem}

We start the proof with a lemma.

\begin{lemma}\label{lem:maxarea}
For any $K \in \K^{\rm sym}_S$ and any convex $N$-gon $P$ of maximal area in $K$, where $N \geq 3$, we have $c \in P$.
\end{lemma}

\begin{proof}
Consider an $N$-gon $P$ contained in $K$ that does not contain $c$. Then there is a great circle $L$ through $c$ disjoint from $P$. Let the endpoints of $L \cap K$ be denoted by $p$ and $q$. 
Since $K \in \K^{\rm sym}_S$, there is a great circle $L$ through $c$, and two distance curves $C_p$ and $C_q$ with axis $L$ and that pass through $p$ and $q$, respectively, which touch $\bd(K)$. Here, without loss of generality, we may assume that $L$ is the line $L$ in the coordinate system in the last paragraph of Subsection~\ref{subsec:prelim}. We denote by $K_0$ the intersection of $K$ with the closed hemisphere of $\Sph^2$ bounded by $L$ and containing $P$. Observe that the interior of this hemisphere contains $P$.

Let $C$ be any distance curve with axis $L$ that intersects $K$. Then $C$ intersects $[p,q]_s$. Since $P$ is compact, there is some $\varepsilon > 0$ such that the angular distance of $C \cap P$ from $L$, measured on $C$, is at least $\varepsilon$, implying that the angular length of $C \cap K_0$ is greater by $\varepsilon$ than the angular length of $C \cap P$.
Let $P'$ be the rotated copy of $P$ about the poles of $\Sph^2$ by $\varepsilon$ towards $L$. Then $P' \subseteq K_0 \subset K$, and at least one vertex of $P'$ is contained in the interior of $K$. This yields that the area of $P$ is not maximal.
\end{proof}

\begin{proof}[Proof of Theorem~\ref{spherical-macbeath-thm}]
First, note that by the spherical variant of Dowker's theorem (c.f., e.g., \cite{FejesTothFodor}), any $N$-gon of maximal area inscribed in $B_c(r)$ is a regular $N$-gon. First, we show that the area of such a regular $N$-gon is equal to $C(r,N)$.

Consider a spherical right triangle whose hypotenuse is $r$, and one of its angles is $\frac{\pi}{N}$. Let the third angle of the triangle be $\gamma$. Applying the spherical cosine law for this triangle for three angles and a side, we obtain that $\gamma = \arctan \frac{\cos \frac{\pi}{N}}{ \sin \frac{\pi}{N} \cos r}$. Thus, by Girard's theorem, the area of this triangle is $\frac{\pi}{N}+ \arctan \frac{\cos \frac{\pi}{N}}{ \sin \frac{\pi}{N} \cos r} - \frac{\pi}{2}$. Then our statement follows by multiplying this quantity by $2N$. This proves the equality in (\ref{eq:Sas}).

 To prove the inequality in (\ref{eq:Sas}), we choose an arbitrary symmetrization $\sigma_{L,H} \in \Sigma_c$, and observe that by Theorem~\ref{general-ineq}, it suffices to show that for any $K \in \K^{\rm sym}_S$ and $m \geq 3$, we have $A_N(K) \geq A_N(\sigma_{L,H}(K))$.

\begin{figure}[ht]
\begin{center}
\includegraphics[width=0.45\textwidth]{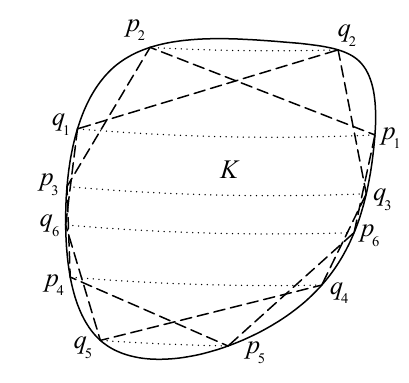}
\caption{An illustration for the proof of Theorem~\ref{spherical-macbeath-thm} with $N=6$. The dotted arcs denote the chords of $K$ on the distance curves $C_i$. The dashed polygonal curves denote the boundaries of the spherical inscribed polygons $P$ and $Q$.}
\label{fig:inscribed}
\end{center}
\end{figure}

Set $K'= \sigma_{L,H}(K)$, and let $P'$ be an $N$-gon contained in $K'$ of area $A_N(K')$. Then $P'$ is inscribed in $K'$, and since $K'$ is axially symmetric to $H$, the reflected copy $Q'$ of $P$ about $H$ is also a maximal area $N$-gon in $K'$. Let $p_1', p_2', \ldots, p_N'$ be the vertices of $P'$ in counterclockwise order, where we set $p_{N+1}'=p_1'$. We follow the idea of Macbeath in \cite{Macbeath} where for every $p_i'$ we define $C_i$ as the distance curve with axis $L$, passing through $p_i'$. Then $C_i \cap K'$ is a circular arc with endpoints $p_i'$ and $q_i'$, where $q_i'$ is the reflection of $p_i'$ to $H$, and thus, it is a vertex of $Q'$.
The curve $C_i$ intersects $K$ in a rotated copy of the circular arc $C_i \cap K'$. We denote the endpoints of this arc by $p_i$ and $q_i$, where the notation is chosen such that the rotation that moves $p_i'$ into $p_i$ moves $q_i'$ into $q_i$ (see Figure~\ref{fig:inscribed}). Furthermore, we define $P$ and $Q$ as the convex hulls of the $p_i$ and the $q_i$, respectively.

In the remaining part, we show that 
\[
\area(P) + \area(Q) \geq \area(P')+\area(Q').
\]
We establish this inequality under the assumption that the two distance curves with axis $L$ that touch $P'$ and $Q'$,  touch them at vertices. If one or both of them touch $P'$ and $Q'$ at an interior point of a side, a similar consideration can be applied.

For $i=1,2,\ldots,N$, we define the quantities $A_i'$ as follows. Consider the case that $[p_i',q_{i+1}']_s$ and $[q_i',p_{i+1}']_s$ do not cross. Then $A_i'$ denotes the area of the region in $S$ bounded by $[p_i',q_{i+1}']_s$, $[q_i',p_{i+1}']_s$, and the circular arcs $C_i \cap K'$ and $C_{i+1} \cap K'$. Assume that $[p_i',q_{i+1}']_s$ and $[q_i',p_{i+1}']_s$ cross at a point $x$. For $j=i,i+1$, let $R_j$ denote the region bounded by $[p_j',x]_s$, $[q_j',x]_s$ and $C_j \cap K'$. Then one of $R_i$ and $R_{i+1}$ is covered by both $P'$ and $Q'$, and the other one overlaps neither $P'$ nor $Q'$. Let the first region be denoted by $R_+$, and the second one by $R_-$. Then we set $A_i'=\area(R_+)-\area(R_-)$.
Finally, we observe that $[p_i',q_{i+1}']_s$ and $[q_i',p_{i+1}']_s$ cross if and only if $[p_i,q_{i+1}]_s$ and $[q_i,p_{i+1}]_s$ cross, and define the quantities $A_i$ analogously to the $A_i'$.

Now, we have
\[
\area(P')+\area(Q')=\sum_{i=1}^m A_i', \quad \area(P)+\area(Q)=\sum_{i=1}^N A_i.
\]
Thus, it is sufficient to prove that for any value of $i$, $A_i' \geq A_i$, which we are going to do in the remaining part of the proof.

Consider some $i$ such that $[p_i',q_{i+1}']_s$ and $[q_i',p_{i+1}']_s$ cross neither each other, nor $L$. Then applying the idea of the proof of Theorem~\ref{thm:planarconvexity}, we obtain that $A_i \geq A_i'$.

Assume that $[p_i',q_{i+1}']_s$ and $[q_i',p_{i+1}']_s$ cross $L$ but not each other. Without loss of generality, we may assume that the distance of $C_i$ from $L$ is not greater than that of $C_{i+1}$. Thus, by the symmetry and angular monotonicity property of $K$, the angular length of $C_i \cap K'$ is not less than that of $C_{i+1} \cap K'$. We follow the idea of the proof of Theorem~\ref{thm:planarconvexity}, and, without loss of generality, assume that $C_{i+1} \cap K'$ is a single point, i.e., $p_{i+1}'=q_{i+1}'$, which yields also that $p_{i+1}=q_{i+1}$. But then the Lexell circle of $[p_i',q_i']_s$ through this point separates $C_{i+1}$ from $[p_i',q_i']_s$, implying that $\area(\conv_s \{ p_i,q_i,p_{i+1} \}) \geq \area(\conv_s \{ p_i',q_i',p_{i+1}' \})$ and thus, $A_i \geq A_i'$.

Consider the case that $[p_i',q_{i+1}']_s$ and $[q_i',p_{i+1}']_s$ cross each other, but not $L$. Without loss of generality, we assume that $C_i \cap K'$ lies in the boundary of the region $R_+$ of positive weight. By Lemma~\ref{lem:maxarea} and the symmetry of $K'$, the distance of $C_i$ from $L$ is not greater than that of $C_{i+1}$, as otherwise $c \notin P$. 

Assume that $C_i \cap \sigma_{L,H}(K)$ belongs to the part with positive weight. Then by Lemma~\ref{lem:maxarea} and the symmetry of $\sigma_{L,H}(K)$,  $C_i$ is not closer to $G$ than $C_i$, as otherwise $c \notin P$. Thus, as in the previous cases, applying a suitable rotation to $[p_i',p_{i+1}']_s$ and $[q_i',q_{i+1}']_s$, we may assume that $q_{i+1}'=p_{i+1}'$, implying that $q_{i+1}=p_{i+1}$. As before, we have $\area (\conv_s \{ p_i,q_i,p_{i+1} \}) \geq \area (\conv_s \{ p_i',q_i',p_{i+1}' \})$, which yields that $A_i \geq A_i'$.

Finally, if $[p_i',q_{i+1}']_s$ and $[q_i',p_{i+1}']_s$ cross each other and also $L$, to obtain the inequality $A_i \geq A_i'$ we can combine the arguments in the two previous cases.
\end{proof}

\begin{problem}
    Remove the symmetry assumption on the convex disk $K$ in Theorem~\ref{spherical-macbeath-thm}.
\end{problem}

\subsection{An isoperimetric-type inequality for floating areas of spherically convex disks}\label{subsec:floating}

Let $\mathcal{K}^n$ denote the set of all convex bodies in $\R^n$. For $K\in\mathcal{K}^n$, the generalized Gaussian curvature at $x\in\bd(K)$ is denoted by $H_{n-1}(K,x)$. Note that $H_{n-1}(K,x)$ exists for $\mathcal{H}^{n-1}$-almost all $x\in\bd(K)$ \cite[Lemma 2.3]{Hug-1996}. A quantity which plays a key role in affine differential geometry is the \emph{affine surface area} of $K$, which is defined as
\[
\as(K):=\int_{\bd(K)}H_{n-1}(K,x)^{\frac{1}{n+1}}\,d\mathcal{H}^{n-1}(x).
\]
Affine surface area is invariant under volume-preserving affine transformations, and it is an upper semicontinuous valuation on the set of convex bodies in $\mathbb{R}^n$ which vanishes on polytopes. For more background on affine surface area  and its many applications to convexity, we refer the reader to, e.g., \cite[Section 10.5]{SchneiderBook}.

Interestingly, the quantity $\as(K)$ is closely related to how well $K$ can be approximated by polytopes, as follows.
For any $K\in\mathcal{K}^n$, let $\mathcal{P}_N(K)$ denote the set of all polytopes contained in $K$ with at most $N$ vertices, and set
\begin{equation}\label{best-polytope}
\dist(K,\mathcal{P}_N(K)):=\inf\left\{\vol_n(K)-\vol_n(P):\, P\in\mathcal{P}_N(K)\right\}.
\end{equation}
By a compactness argument, there exists a \emph{best-approximating polytope} $\widehat{P}\in\mathcal{P}_N(K)$ which attains the infimum in \eqref{best-polytope}.

If $K\in\mathcal{K}^n$ has $C^2$ boundary and positive Gaussian curvature, then there exists a constant $\del_{n-1}$ (whose value depends only on the dimension $n$) such that
\begin{equation}\label{main-asymptotic}
\lim_{N\to\infty}\dist(K,\mathcal{P}_N(K))N^{\frac{2}{n-1}} = \frac{1}{2}{\rm del}_{n-1}\left[{\rm as}(K)\right]^{\frac{n+1}{n-1}}.
\end{equation}
The quantity $\del_{n-1}$ is called the \emph{Delone triangulation number} in $\R^n$, and it may be defined via \eqref{main-asymptotic}. These numbers are connected with Delone triangulations in $\R^{n-1}$; for more background, we refer the reader to, e.g.,  \cite{MaS1,MaS2} and the references therein. Only the first few values of $\del_{n-1}$ are known; in particular, the exact value ${\rm del}_1=1/6$ was stated by L. Fejes T\'oth \cite{lagerungen} and proved by McClure and Vitale \cite{McClure-Vitale-1975}, and Gruber \cite{Gruber:1988} proved that ${\rm del}_2=1/(2\sqrt{3})$. The best known asymptotic estimate for $\del_{n-1}$ is due to Mankiewicz and Sch\"utt \cite{MaS1,MaS2}, who proved that
\[
{\rm del}_{n-1}=\frac{n}{2\pi e}\left(1+O\left(\frac{\ln n}{n}\right)\right).
\]

The spherical analogue of affine surface area for a spherically convex body $K$ is the so-called \emph{floating area} of $K$, defined by Besau and Werner \cite{Besau-Werner} as
\begin{equation}\label{eq:floatingarea}
\Omega_s(K) = \int_{\partial K}H_{n-1}^{\mathbb{S}^n}(K,x)^{\frac{1}{n+1}}\,d\mathcal{H}^{n-1}(x).
\end{equation}
The function $\Omega_s(\cdot)$ is an upper semicontinuous valuation on the set $\mathcal{K}(\mathbb{S}^n)$ of spherically convex bodies, and it vanishes on spherical polytopes \cite{Besau-Werner}.

The following conjecture is due to Besau and Werner \cite[Conjecture 7.5]{Besau-Werner}.

\begin{conjecture}[Isoperimetric Inequality for Floating Area]\label{conj:BW}
    Let $K\in\mathcal{K}(\mathbb{S}^n)$. Then $\Omega_s(K)\leq \Omega_s(B(K))$, where $B(K)$ is a spherical ball such that $\vol_s(K)=\vol_s(B(K))$. Equality holds if and only if $K$ is a spherical cap.
\end{conjecture}

Very recently, this conjecture was verified by Besau and Werner \cite[Theorem 4.1]{Besau-Werner-Lp-2024} for $K\in\mathcal{K}(\mathbb{S}^n)$, $n\geq 3$, under  the  restrictions that the origin is the ``GHS-center" of $K$ and the gnomonic projection of $K$ is contained in the ball of radius $\sqrt{n(n-2)}$. We refer the reader to \cite{Besau-Werner-Lp-2024} for the precise statements. In this subsection, we study the missing planar case and prove that the inequality in Conjecture~\ref{conj:BW} holds for smooth and $c$-symmetric convex disks in $\mathbb{S}^2$. In what follows, given $p\in\mathbb{S}^n$, let $\mathcal{K}_p(\mathbb{S}^n):=\{K\in\mathcal{K}(\mathbb{S}^n):\, p\in K,\, K\subset S\}$.

\begin{theorem}\label{isop-ineq-floating-area}
    Let $K\in\mathcal{K}_p(\mathbb{S}^2)$ be $c$-symmetric  and have $C^2_+$ boundary. Then $\Omega_s(K)\leq \Omega_s(D_c(K))$, where $D_c(K)$ is the spherical disk centered at $c$ such that $\area(D_c(K))=\area(K)$.
\end{theorem}

Unfortunately, our proof method does not characterize the equality case in Theorem~\ref{isop-ineq-floating-area}.


\vspace{2mm}

In analogy with the Euclidean case, the floating area $\Omega_s(K)$ is closely related to how well a spherically convex body $K\subset\mathbb{S}^n$ can be approximated by spherical polytopes. The connection between the floating area and the asymptotic best and random approximation of convex bodies in spherical and hyperbolic spaces was  studied by Besau, Ludwig and Werner in \cite{BLW}. The next result is an asymptotic formula for the best approximation of spherically convex bodies in $\mathbb{S}^n$ by inscribed spherical polytopes with at most $N$  vertices. To state the result, for $K\in\mathcal{K}_p(\mathbb{S}^n)$, let $\mathcal{P}_N^{\mathbb{S}^n}(K)$ denote the set of all spherical polytopes inscribed in $K$ with at most $N$ vertices, and set 
\[
\dist_s(K,\mathcal{P}_N^{\mathbb{S}^n}(K)):=\inf\{\vol_s(K)-\vol_s(P):\, P\in\mathcal{P}_N^{\mathbb{S}^n}(K)\}.
\]
Also note that by a compactness argument, there exists a best-approximating polytope $\widehat{P}\in\mathcal{P}_N^{\mathbb{S}^n}(K)$ such that  $\dist_s(K,P_N^{\mathbb{S}^n}(K))=\vol_s(K)-\vol_s(\widehat{P})$.

\begin{theorem}\label{BLW-thm}
    For every $K\in\mathcal{K}(\mathbb{S}^n)$ with $C^2_+$ boundary, we have 
    \[
\lim_{N\to\infty}\dist_s(K,\mathcal{P}^{\mathbb{S}^n}_N(K))N^{\frac{2}{n-1}}=\frac{1}{2}{\rm del}_{n-1}\left[\Omega_s(K)\right]^{\frac{n+1}{n-1}}.
    \]
\end{theorem}

\begin{proof}
The proof is a suitable modification of the proof of \cite[Theorem 2.4]{BLW}, and thus, we only sketch it. First, we derive an asymptotic formula for the weighted volume difference of any convex body $C\in\mathcal{K}^n$ with $C^2_+$ boundary and a minimizing inscribed polytope with $N$ vertices  as $N\to \infty$.  To do so, we can use a formula of Glasauer and Gruber \cite{glasgrub} for the weighted volume difference, combined with adaptations of the arguments of Ludwig \cite{Ludwig1999} from the general unconstrained case of weighted best approximation to the inscribed case. As a result, we only need to replace ${\rm ldel}_{n-1}$ with ${\rm del}_{n-1}$ in \cite[Theorem 3]{Ludwig1999} to obtain the desired formula for weighted inscribed approximation in the Euclidean case. 

To obtain the desired formula for the inscribed approximation of spherical convex bodies contained in an open hemisphere of class $C^2_+$, we can now argue as Besau, Ludwig and Werner do in the proof of \cite[Theorem 2.4]{BLW}, where the more general case in which the polytopes are unconstrained was settled. Simple modifications to their arguments yield the result, only now we use the inscribed weighted approximation result, rather than the unconstrained one. \end{proof}

Choosing $n=2$ in Theorem \ref{BLW-thm} and using ${\rm del}_1=1/6$, we obtain
   \begin{equation}\label{special-case}
\lim_{N\to\infty}\dist_s(K,\mathcal{P}^{\mathbb{S}^2}_N(K))N^2=\frac{1}{12}\Omega_s(K)^3.
    \end{equation}
    We are now in position to prove Theorem \ref{isop-ineq-floating-area}.

\begin{proof}[Proof of Theorem \ref{isop-ineq-floating-area}]
    Let $K\in\mathcal{K}_p(\mathbb{S}^2)$ be $c$-symmetric with $C^2_+$ boundary. Combining \eqref{special-case} and Theorem \ref{spherical-macbeath-thm},  we get
    \begin{align*}
        \frac{1}{12\area(K)}\Omega_s(K)^3&=\lim_{N\to\infty}\frac{\dist_s(K,\mathcal{P}_N^{\mathbb{S}^2}(K))}{\area(K)}N^2\\
        &\leq \lim_{N\to\infty}\frac{\dist_s(D_c(K),\mathcal{P}_N^{\mathbb{S}^2}(D_c(K)))}{\area(D_c(K))}N^2
        =\frac{1}{12\area(D_c(K))}\Omega_s(D_c(K))^3.
    \end{align*}
    Since $\area(K)=\area(D_c(K))$, the result follows.
\end{proof}

\begin{remark}
    For $p\in\R\setminus\{-n\}$ and $K\in\mathcal{K}(\mathbb{S}^n)$, Besau and Werner \cite{Besau-Werner-Lp-2024} very recently defined the \emph{$L_p$ floating area} of $K$ by
    \[
    \Omega_{s,p}(K)=\int_{\partial K}H_{n-1}^{\mathbb{S}^n}(K,x)^{\frac{p}{n+p}}\,d\mathcal{H}^{n-1}(x),
    \]
    where one assumes that $K$ is of class $C_+^2$ for $p<0$. This may be regarded as a spherical analogue of the (Euclidean) $L_p$ affine surface area; see \cite{Besau-Werner-Lp-2024}. In \cite[Corollary 6.4]{Besau-Werner-Lp-2024}, they proved that for $p\geq 1$, $n\geq 3$ and all  $K\in\mathcal{K}(\mathbb{S}^n)$ whose ``GHS-center" is the origin and  whose gnomonic projection is contained in the ball of radius $\sqrt{n(n-2)}$, 
    \begin{equation}\label{Lp-inequality}
        \Omega_{s,p}(K) \leq \Omega_{s,p}(B(K))
    \end{equation}
    with equality if and only if $K$ is a geodesic ball. Using \cite[Theorem 6.3]{Besau-Werner-Lp-2024} and Theorem \ref{isop-ineq-floating-area}, it follows  that inequality \eqref{Lp-inequality} also holds in the planar case $n=2$ for all $p\geq 1$ and all $K\in\mathcal{K}_c(\mathbb{S}^2)$ which are $c$-symmetric and have $C_+^2$ boundary. Our method does not characterize the equality conditions.
\end{remark}

\section{A spherical version of the Winternitz theorem}\label{sec:Winternitz}

Let $K$ be a convex body in $\R^2$. The \emph{centroid} $g(K)$ of $K$ is the point
\[
g(K) = \frac{1}{\area(K)}\int_K x\,dx.
\]
Note that $g(K)$ lies in the interior of $K$. Let $L$ be any line passing through $g(K)$. Then $L$ cuts $K$ into two convex bodies $K_1$ and $K_2$, contained in opposite closed half-planes bounded by $L$. The following result is due to Winternitz (for a proof, see \cite[pp. 54--55]{Blaschke-1923}).
 
\begin{theorem}[Winternitz]\label{winternitz}
    \begin{equation}
    \frac{4}{5}\leq\frac{{\rm area}(K_1)}{{\rm area}(K_2)} \leq \frac{5}{4},
\end{equation}
with equality if and only if $K$ is a triangle and $L$ is parallel to one side of the triangle.
\end{theorem}
Winternitz's theorem was rediscovered by many authors, including \cite{Ehrhart-1955,LL1935,Neumann-1945,Newman-1958,YB-1961}.  Gr\"unbaum \cite{Grunbaum-1960} extended Theorem~\ref{winternitz} to convex bodies in all dimensions. Generalizations of Gr\"unbaum's result to sections and projections of convex bodies were proved in, e.g., \cite{FMY-2017,MNRY,MSZ,SZ-2017}. For other generalizations of Winternitz's theorem within geometry, see also \cite{BCHM-2011,Shyntar-Yaskin}. Winternitz's theorem is closely connected to the Winternitz measure of symmetry, which has recently found applications to data depth in statistics \cite{NSW-2019}.  

In this section, we shall obtain a spherical analogue of Winternitz's theorem for spherically convex disks. First, we will need the following definition of a spherical centroid for certain subsets of $\mathbb{S}^n$. Definition~\ref{defn:sphcentroid} can be found in the very recent paper \cite{BHPS} of Besau, Hack, Pivovarov and Schuster, although it seems that the same concept had been known for a long time before and was used in engineering literature  (see, e.g., \cite{Brock,Bjerre, Fog,Galperin}).

\begin{definition}\label{defn:sphcentroid}
    For a finite set $\{u_1,\ldots,u_N\}\subset\mathbb{S}^n$ and a Borel set $A\subset\mathbb{S}^n$ with $\sigma(A)>0$, their respective \emph{spherical centroids} are defined by
    \[
g_s(\{u_1,\ldots,u_N\}) := \frac{\sum_{i=1}^N u_i}{\left\|\sum_{i=1}^N u_i\right\|_2}
\qquad \text{and} \qquad
g_s(A):=\frac{\int_A u\,d\sigma(u)}{\left\|\int_A u\,d\sigma(u)\right\|_2},
    \]
    provided the denominators are nonzero.
\end{definition}

Imagine $\mathbb{S}^n$ as the unit sphere of $\R^{n+1}$ centered at $o$. It is well known that for any set $A$ in $\R^{n+1}$, the first moment of the set with respect to any hyperplane $H$ containing the centroid is zero. Now imagine $A \subset \mathbb{S}^n$ with $\sigma(A) > 0$ as a `hypersurface' in $\R^{n+1}$. Note that then $g_s(A)$ is the radial projection of the Euclidean centroid $g(A)$ of $A$ to the sphere. Thus, the Euclidean first moment of $A$ with respect to any hyperplane $H$ containing $o$ and $g(A)$ is zero. Note also that if $H$ is such a hyperplane, and it intersects $\mathbb{S}^{n}$ in the great sphere $G$, then the Euclidean (signed) distance of a point $p \in \mathbb{S}^n$ from $H$ is equal to $\sin \theta_p$, where $\theta_p$ denotes the (signed) spherical distance of $p$ from $G$. This leads to the following.

\begin{remark}\label{rem:sphlever}
For any great sphere $H$ of $\mathbb{S}^n$ and point $p \in \mathbb{S}^n$, let $\theta_H(p)$ denote the signed spherical distance of $p$ from $G$ with respect to some fixed orientation of $\mathbb{S}^n$. Then, for any Borel set $A\subset\mathbb{S}^n$ with $\sigma(A)>0$, and for any great sphere $H$ of $\mathbb{S}^n$ , the equality
\[
0 =  \iint_{A} \sin \theta_H(p) \, d \sigma.
\]
holds if and only if $H$ contains $g_s(A)$.
The same argument shows that this observation holds also for finite point systems on the sphere.
\end{remark}

Remark~\ref{rem:sphlever} yields the `spherical rule of the lever' that can be found in \cite[Eq. 17]{Galperin}. Based on this, for any great sphere $H$ and Borel set $A$, we set
\[
M_H(A) = \iint_{A} \sin \theta_H(p) \, d \sigma.
\]

In the following list we collect some properties of spherical centroids and spherical moments.
\begin{enumerate}
\item[(1)] For any great sphere $H \subset \Sph^n$ and Borel sets $A, B \subset \Sph^n$ satisfying $\inter(A)\cap\inter(B)=\emptyset$, we have $M_{H}(A\cup B) = M_{H}(A)+M_{H}(B)$.
\item[(2)] If $K \subset \Sph^n$ is a spherical convex body, then $g_s(K) \in \inter (K)$.
\end{enumerate}

Before stating our main result, we prove a lemma.

\begin{lemma}\label{lem:sphcentroid}
Let $K \subset \Sph^2$ be a convex disk with piecewise smooth boundary. Then $K$ is contained in the open hemisphere with the centroid $g_s(K)$ of $K$ as its center.
\end{lemma}

\begin{proof}
For any smooth point $p \in \bd (K)$, let $N(p)$ denote the inner unit normal of $K$, if it exists. Then, by \cite{AD2004},
\[
2 \iint_K p \, d \sigma = \int_{\bd (K)} N(p) \, ds,
\]
where $ds$ denotes integration with respect to arclength. By Definition~\ref{defn:sphcentroid}, this equality can be written as
\[
g_s(K) = \frac{1}{2 \area(K)} \int_{\bd (K)} N(p) \, ds.
\]
Let $K' = \{ q \in \Sph^2 : \langle p, q \rangle \geq 0 \hbox{ for every } p \in K \}$. Then $K'$ is the set of points $q\in\Sph^2$ with the property that the closed hemisphere centered at $q$ contains $K$. We note that $-K'$ is the spherical polar of $K$, and thus, $K'$ is a convex disk. Furthermore, for every smooth point $p \in \bd(K)$, $N(p)$ is a boundary point of $K'$. This yields that $g_s(K)$ can be written as an integral average of boundary points of $K'$. This and the fact that $K'$ has nonempty interior proves the assertion.
\end{proof}

Our main result of Section \ref{sec:Winternitz} is as follows.

\begin{theorem}\label{thm:sphWinternitz}
Let $K \subset \Sph^2$ be a convex disk with centroid $c$. Let $L$ be a great circle through $c$, and let $S_1$ and $S_2$ denote the closed hemispheres bounded by $L$. Let $K_1 = K \cap S_1$ and $K_2 = K \cap S_2$. Then there is a triangle $T$ with $c$ as its centroid such that, with the notation $T_1 = T \cap S_1$ and $T_2 = T \cap S_2$, we have
\[
\frac{\area(K_2)}{\area(K_1)} \leq \frac{\area(T_2)}{\area(T_1)}.
\]
\end{theorem}

\begin{proof}
For convenience, we assume that $K$ is smooth and strictly convex. This, by Lemma~\ref{lem:sphcentroid},  implies that $K$ is contained in the open hemisphere centered at $c$. Let $[q_1,q_2]_s = K \cap L$  and let $\hat{S}_1 = S_1 \setminus L$. Then  for every $p \in \hat{S}_1$, the triangle $T(p)$ with vertices $\{ p, q_1, q_2 \}$ exists for any $p \in \hat{S}_1$. Furthermore, the function $M_L(T(p))$ is a $C^{\infty}$-class function of $p$ on $\hat{S}_1$. Since $T(p_1) \subsetneq T(p_2)$ implies $M_L(T(p_1)) < M_L(T(p_2))$, the Jacobian of $f(p) = M_L(T(p))$ is not zero at any $p \in \hat{S}_1$. Furthermore, since $K_1 \subset S_1$, and for a suitable choice of $p$, $f(p)$ can get arbitrarily close to $M_L(S_1)$, there are points $p \in \hat{S}_1$ such that $f(p) = M_L(K_1)$. Combining these observations, we obtain that for any open half-great circle in $\hat{S}_1$, connecting the midpoint of $[q_1,q_2]_s$ and its antipodal point, there is a unique point $p$ with $f(p) = M_L(K_1)$. By continuity, these points belong to a continuous curve $\Gamma$. Note that no point $p$ of $\Gamma$ belongs to $K_1$, as otherwise $T(p) \subsetneq K_1$. Thus, for any $p \in \Gamma$, there are unique points $x_1, x_2$ such that for $i=1,2$, $[p,q_i]_s \cap K_1 =[p,x_i]_s$. By continuity, we can choose $p$ in such a way that $x_1$ and $x_2$ have the same distance $\theta$ from $L$. We choose such a point $p$, and set $T_1 = T(p)$.
Now observe that, apart from $x_1,x_2$, the distance of any point of $T_1 \setminus K_1$ from $L$ is strictly greater than $\theta$, and for any point of $K_1 \setminus T_1$ it is at most $\theta$. 
In the notation of Remark \ref{rem:sphlever}, this means that $\theta_{L}\left(g_s(K_1\setminus T_1)\right)<\theta_L(g_s(T_1\setminus K_1))$. Since $M_L(T_1 \setminus K_1) = M_L(K_1 \setminus T_1)$ by our construction, we have $\area(T_1 \setminus K_1) < \area(K_1 \setminus T_1)$; that is, $\area(T_1) < \area(K_1)$.

In the following, for $i=1,2$, let $L_i$ be the great circle spanned by $[p,q_i]_s$, and let $X$ be the lune bounded by $L_1, L_2$ and containing $T_1$.  For $i=1,2$, we set $X_i = X \cap S_i$, and observe that by convexity, $X_2$ is a spherical triangle containing $K_2$, and hence, $|M_L(X_2)| > |M_L(K_2)|$. We note also that $T_1 = X_1$ and $M_L(T_1) = |M_L(K_2)|$ yields that $X_2$ contains the center of $X$. Now, for any $0 < \theta'<\frac{\pi}{2}$, let $Y_{\theta'}$ denote set of points of $X$ whose spherical distance from $L$ is at most $\theta'$. Then, by convexity, $\bd (K_2)$ contains exactly two points at distance $\theta'$ from $L$. Let $y_1$ and $y_2$ denote these points, and let $L'$ denote the great circle through $y_1, y_2$. Then $[y_1,y_2]_s = L' \cap K_2$. Let $z_i$ be the intersection of $L_i$ and $L'$ on the boundary of $X_2$. Let $T_2(\theta')$ denote the spherical quadrangle $\conv_s \{ q_1,q_2,z_1,z_2\}$. In the following, we choose the value of $\theta'$ such that $M_L(T_2(\theta'))=M_L(K_2)$, and set $T_2 = T_2(\theta')$. Note that $T_2 \setminus K_2 \subseteq Y_{\theta'} \setminus K_2$ and $K_2 \setminus Y_{\theta'} \subseteq K_2 \setminus T_2$. Thus, apart from $y_1,y_2$, any point of $T_2 \setminus K_2$ is closer to $L$ than any point of $K_2 \setminus T_2$. This implies that $\area(T_2) > \area(K_2)$. Thus, we have
\[
\frac{\area(K_2)}{\area(K_1)} \leq \frac{\area(T_2)}{\area(T_1)}.
\]

Note that by our construction and Remark~\ref{rem:sphlever}, the centroid of $T$ lies on $L$. Let $\bar{T}$ denote the rotated copy of $T$ about the two poles of $L$ 
with the property that the centroid of $\bar{T}$ is $c$. Then $\bar{T}$ satisfies the conditions of the theorem.
\end{proof}

\section*{Acknowledgments}

We are  grateful to the anonymous referee for their careful reading and comments which helped improve this paper. SH would also like to thank Florian Besau and Vlad Yaskin for the valuable discussions on the topic of the paper.

\bibliographystyle{plain}
\bibliography{main_v5.bib}






\end{document}